\documentclass[11pt,reqno]{article}

\usepackage[usenames]{color}
\usepackage{amssymb}
\usepackage{amsmath}
\usepackage{amsthm}
\usepackage{amsfonts}
\usepackage{amscd}
\usepackage{graphicx}

\usepackage[noadjust]{cite}

\usepackage[colorlinks=true,
linkcolor=webgreen,
filecolor=webbrown,
citecolor=webgreen]{hyperref}

\definecolor{webgreen}{rgb}{0,.5,0}
\definecolor{webbrown}{rgb}{.6,0,0}
\usepackage{color}
\usepackage{fullpage}
\usepackage{float}

\usepackage{cases}

\usepackage{graphics,amsmath,amssymb}
\usepackage{amsthm}
\usepackage{amsfonts}
\usepackage{latexsym}

\begin{document}

\theoremstyle{plain}
\newtheorem{theorem}{Theorem}
\newtheorem{lemma}{Lemma}
\newtheorem{corollary}{Corollary}

\newtheorem{proposition}{Proposition}
\newtheorem{example}{Example}
\newtheorem{definition}{Definition}

\theoremstyle{definition}
\newtheorem{remark}{Remark}

\renewcommand{\thefootnote}{\fnsymbol{footnote}}
\setlength{\skip\footins}{1cm}
\setlength{\footnotesep}{0.3cm}

\begin{center}
\vskip 1cm{\Large\bf Formulas involving Cauchy polynomials, Bernoulli	
\vskip .06in polynomials, and generalized Stirling numbers of both kinds\footnote[1]{
{\normalsize This paper has been published in \emph{Axioms} \textbf{2025}, \emph{14}, 746.
\url{https://doi.org/10.3390/axioms14100746}}}
}
\vskip .2in \large Jos\'{e} Luis Cereceda \\
{\normalsize Collado Villalba, 28400 (Madrid), Spain} \\
\href{mailto:jl.cereceda@movistar.es}{\normalsize{\tt jl.cereceda@movistar.es}}
\end{center}

\begin{abstract}
In this paper, we derive novel formulas and identities connecting Cauchy numbers and polynomials with both ordinary and generalized Stirling numbers, binomial coefficients, central factorial numbers, Euler polynomials, $r$-Whitney numbers, and hyperharmonic polynomials, as well as Bernoulli numbers and polynomials. We also provide formulas for the higher-order derivatives of Cauchy polynomials and obtain corresponding formulas and identities for poly-Cauchy polynomials. Furthermore, we introduce a multiparameter framework for poly-Cauchy polynomials, unifying earlier generalizations like shifted poly-Cauchy numbers and polynomials with a $q$ parameter.
\end{abstract}

\section{Introduction}\label{sec:1}

Cauchy polynomials of the first and second kind, $c_n(x)$ and $\widehat{c}_n(x)$, are defined by their respective generating functions (see~\cite{cheon} (Section 5) and~\cite{mezo} (Equations~(1) and (2))):
\begin{equation}\label{gf1}
\frac{t}{(1+t)^x \ln(1+t)} = \sum_{n=0}^{\infty} c_n(x) \frac{t^n}{n!}  \quad (|t| < 1),
\end{equation}
and
\begin{equation}\label{gf2}
\frac{t (1+t)^x}{(1+t) \ln(1+t)} = \sum_{n=0}^{\infty} \widehat{c}_n(x) \frac{t^n}{n!}  \quad (|t| < 1).
\end{equation}
(Note that $t$ is replaced by $-t$ in the definition of $\widehat{c}_n(x)$ in~\cite{cheon} (Theorem 5.2)). From \eqref{gf1} and~\eqref{gf2}, it follows that $c_n(x)$ and $\widehat{c}_n(x)$ are related by $\widehat{c}_n(x) = c_n(1-x)$ or $c_n(x) = \widehat{c}_n(1-x)$. When $x=0$, $c_n = c_n(0)$ and $\widehat{c}_n = \widehat{c}_n(0)$ are the classical Cauchy numbers of the first and second kind, respectively. Some properties of these numbers are recorded in Exercise 13 on p.~293 of Comtet's influential book~\cite{comtet}, with~many more established in Merlini~et~al.'s seminal paper~\cite{merlini}. We additionally note that the numbers $c_n/n!$ are sometimes referred to as Bernoulli numbers of the second kind (see, e.g.,~\cite{agoh,qi}).

The Cauchy polynomials $c_n(x)$ and $\widehat{c}_n(x)$ can be defined equivalently by their respective integral representations~\cite{cheon} (Section 5):
\begin{equation}\label{defcaup}
c_n(x) = n! \int_{0}^{1} \binom{t - x}{n} dt \quad \, \text{and} \quad
\widehat{c}_n(x) = n! \int_{0}^{1} \binom{x - t}{n} dt,
\end{equation}
which we will use frequently throughout the paper. The~first few Cauchy polynomials are given by
\begin{align*}
c_0(x) & = 1, \quad c_1 (x) = \frac{1}{2} -x, \quad c_2(x) = -\frac{1}{6} +x^2,
\quad c_3(x) = \frac{1}{4}- \frac{3}{2}x^2 - x^3, \\
c_4(x) & = -\frac{19}{30} +4 x^2 + 4x^3 + x^4, \quad c_5(x) = \frac{9}{4} -15x^2 -\frac{55}{3} x^3 -\frac{15}{2}x^4 -x^5, \\
c_6(x) & = -\frac{863}{84} + 72x^2 +100x^3 + \frac{105}{2}x^4 +12x^5 + x^6,
\end{align*}
and
\begin{align*}
\widehat{c}_0(x) & = 1, \quad \widehat{c}_1 (x) = -\frac{1}{2} +x, \quad \widehat{c}_2(x) = \frac{5}{6} -2x +x^2,
\quad \widehat{c}_3(x) = -\frac{9}{4} +6x - \frac{9}{2}x^2 + x^3, \\
\widehat{c}_4(x) & = \frac{251}{30} -24x +22 x^2 - 8x^3 + x^4, \,\, \widehat{c}_5(x) = -\frac{475}{12} +120x -125x^2 +\frac{175}{3} x^3 -\frac{25}{2}x^4 +x^5, \\
\widehat{c}_6(x) & = \frac{19087}{84} -720 x +822 x^2 -450x^3 +\frac{255}{2}x^4 -18x^5 + x^6.
\end{align*}

Cauchy numbers and polynomials have been extensively studied and generalized over the years, especially by Professor T. Komatsu. Such generalizations include poly-Cauchy numbers and polynomials~\cite{komatsu,kamano,komatsu4}, poly-Cauchy numbers with a $q$ parameter~\cite{komatsu2,umbralq}, hypergeometric Cauchy numbers and polynomials~\cite{hyper,yuan}, shifted poly-Cauchy numbers~\cite{komatsu5}, incomplete Cauchy and poly-Cauchy numbers and polynomials~\cite{incom,incom2}, higher-order Cauchy numbers and polynomials~\cite{kim}, multiparameter poly-Cauchy numbers and polynomials~\cite{gomaa,young,guettal}, $q$-multiparameter-Cauchy polynomials by Jackson's integrals~\cite{qpoly,qpoly2}, multi-poly Cauchy numbers and polynomials~\cite{kim2,lacpao}, and~higher-level poly-Cauchy numbers~\cite{sirvent}.

This paper is mainly concerned with original Cauchy numbers and polynomials of both kinds, as well as their extension to poly-Cauchy equivalents. As~this paper makes clear, there is still much room for exploring and eventually discovering new fundamental relations and properties for Cauchy numbers and polynomials. Indeed, in~the first part of this article (Sections \ref{sec:2}--\ref{sec:5}), we present a variety of novel formulas and identities relating the Cauchy numbers and polynomials with several families of well-known numbers and polynomials, such as generalized Stirling numbers, the~binomial coefficients, and~Bernoulli numbers and polynomials. We further obtain new recurrence and higher-order derivative formulas for Cauchy polynomials. As~an example, in~Equation~\eqref{hyp5} below, we derive the following recurrence formula for Cauchy polynomials of the second kind:
\begin{equation*}
\frac{\widehat{c}_n(x)}{n!} = \binom{x}{n} + \sum_{m=0}^{n-1} \frac{\widehat{c}_m(x)}{m!}
\, \frac{(-1)^{n-m}}{(n-m)(n+1-m)}, \quad n \geq 1.
\end{equation*}

In the second part of this article (Sections \ref{sec:6} and \ref{sec:7}), we focus on poly-Cauchy numbers and polynomials and their extension to multiparameter poly-Cauchy equivalents. Specifically, in~Section~\ref{sec:6}, we generalize some of the results previously obtained for Cauchy polynomials to poly-Cauchy cases. Furthermore, in~Section~\ref{sec:7}, we define a type of multiparameter poly-Cauchy and multiparameter poly-Bernoulli polynomial, deriving relationships between~them.

\begin{remark}
Due to the generic nature of the results described in this paper, they may be reduced to a number of well-known Cauchy number and polynomial identities. Whenever possible, we will acknowledge these original identities by referencing the appropriate source. For~example, when $x=0$, the~above recurrence formula becomes
\begin{equation*}
\frac{\widehat{c}_n}{n!} = \sum_{m=0}^{n-1} \frac{\widehat{c}_m}{m!} \,
\frac{(-1)^{n-m}}{(n-m)(n+1-m)}, \quad n \geq 1,
\end{equation*}
which reproduces the first identity in Theorem 2.6 of~\cite{merlini}.
\end{remark}

We now introduce generalized Stirling numbers (GSNs for short) of the first and second kind, which play an important role in the Cauchy number and polynomial theory. Following Carlitz~\cite{carlitz} (Equations~(5.2) and (3.2)) and Broder~\cite{broder} (Equations~(56) and (57)) the GSNs of the first and second kind, $\genfrac{[}{]}{0pt}{}{n}{m}_x$ and $\genfrac{\{}{\}}{0pt}{}{n}{m}_x$, are defined by
\begin{equation}
\genfrac{[}{]}{0pt}{}{n}{m}_x = \sum_{i=0}^{n-m} \binom{i+m}{m}
\genfrac{[}{]}{0pt}{}{n}{i+m} x^i, \quad 0 \leq m \leq n, \label{gs1}
\end{equation}
and
\begin{equation}
\genfrac{\{}{\}}{0pt}{}{n}{m}_x = \sum_{i=0}^{n-m} \binom{n}{i}
\genfrac{\{}{\}}{0pt}{}{n-i}{m} x^i, \quad 0 \leq m \leq n, \label{gs2}
\end{equation}
where $x \in \mathbb{R}$, $\genfrac{[}{]}{0pt}{}{n}{m}$ are the (unsigned) Stirling numbers of the first kind, and~$\genfrac{\{}{\}}{0pt}{}{n}{m}$ are those of the~second.

We note that GSNs of the first kind have an alternative representation~\cite{broder} (Theorem~28):
\begin{equation}\label{rep1}
\genfrac{[}{]}{0pt}{}{n}{m}_x = \frac{n!}{m!} \, \frac{d^m}{dx^m} \binom{x+n-1}{n},
\end{equation}
where $\frac{d^m}{dx^m} f(x)$ denotes the $m$-th derivative of $f(x)$ with respect to $x$. For~their part, GSNs of the second kind have an explicit formula~\cite{broder} (Theorem 29):
\begin{equation}\label{gs22}
\genfrac{\{}{\}}{0pt}{}{n}{m}_x = \frac{1}{m!} \sum_{l=0}^m (-1)^{m-l} \binom{m}{l} (x+l)^n.
\end{equation}

When $x=0$, $\genfrac{[}{]}{0pt}{}{n}{m}_0$ and $\genfrac{\{}{\}}{0pt}{}{n}{m}_0$ are ordinary Stirling numbers of the first and second kind, $\genfrac{[}{]}{0pt}{}{n}{m}$ and $\genfrac{\{}{\}}{0pt}{}{n}{m}$. In~general, when $x$ is the non-negative integer $r$, $\genfrac{[}{]}{0pt}{}{n}{m}_r$ and $\genfrac{\{}{\}}{0pt}{}{n}{m}_r$ correspond to $r$-Stirling numbers of the first and second kind, $\genfrac{[}{]}{0pt}{}{n+r}{m+r}_r$ and $\genfrac{\{}{\}}{0pt}{}{n+r}{m+r}_r$, as~introduced by Broder~\cite{broder}. Moreover, the~orthogonality relations for the $r$-Stirling numbers (\cite{broder}, Theorems~5 and 6) extend to the GSNs as follows:
\begin{equation*}
\sum_{l =m}^n (-1)^{n-l} \genfrac{[}{]}{0pt}{}{n}{l}_x  \genfrac{\{}{\}}{0pt}{}{l}{m}_x
= \sum_{l= m}^n (-1)^{n-l} \genfrac{\{}{\}}{0pt}{}{n}{l}_x  \genfrac{[}{]}{0pt}{}{l}{m}_x = \delta_{n,m},
\end{equation*}
where $\delta_{n,m}$ is the Kronecker delta. As~a consequence, the~following inversion formula
\begin{equation}\label{inv}
f_n = \sum_{m=0}^n (-1)^{n-m} \genfrac{[}{]}{0pt}{}{n}{m}_x g_m  \quad \Longleftrightarrow \quad
g_n =  \sum_{m=0}^n \genfrac{\{}{\}}{0pt}{}{n}{m}_x f_m,
\end{equation}
holds true for sequences $\{f_n \}_{n \geq0}$ and $\{g_n \}_{n \geq0}$.

On the other hand, for~arbitrary real numbers $\alpha$ and $x$, the~generalized Bernoulli polynomials $B_n^{(\alpha)}(x)$ (also known as higher-order Bernoulli polynomials) are defined by the generating function (see, e.g.,~\cite{elezovic,boutiche})
\begin{equation}\label{gfgen}
\left( \frac{t}{e^t -1}\right)^{\alpha} e^{xt} = \sum_{n=0}^{\infty} B_n^{(\alpha)}(x) \frac{t^n}{n!} \quad (|t|< 2\pi).
\end{equation}
For $\alpha =1$, $B_n(x) = B_n^{(1)}(x)$ are the classical Bernoulli polynomials. Furthermore, when $x=0$, $B_n^{(\alpha)} = B_n^{(\alpha)}(0)$ are called the Bernoulli numbers of order $\alpha$. Ordinary Bernoulli numbers are then given by $B_n = B_n^{(1)}$, with~$B_0 =1$, $B_1 = -\frac{1}{2}$, $B_2 =\frac{1}{6}$, $B_3 =0$, $B_4 = -\frac{1}{30}$,~etc.

We also explicitly mention the power sum polynomials $S_n(x)$. These can be defined in terms of Bernoulli polynomials as follows~\cite{mikkawy} (Equation~(15)):
\begin{equation}\label{defb}
S_{n-1}(x-1) = \frac{1}{n} \big( B_n(x) - B_n(1) \big), \quad n \geq 1.
\end{equation}
If $m$ is any given positive integer, $S_n(m)$ is equal to the sum of the $n$-th powers of the first $m$ positive integers $1^n +2^n + \cdots + m^n$.

The remainder of this paper is summarized as follows:

In Section~\ref{sec:2}, we derive several formulas and identities involving Cauchy numbers and polynomials, ordinary and generalized Stirling numbers, binomial coefficients, central factorial numbers, Euler polynomials, and~$r$-Whitney numbers. In~Section~\ref{sec:3}, by~using the representation \eqref{rep1}, we establish some formulas for the higher-order derivatives of Cauchy polynomials. In~Section~\ref{sec:4}, we describe a connection between Cauchy polynomials and hyperharmonic polynomials. In~Section~\ref{sec:5}, we obtain explicit formulas for Bernoulli polynomials in terms of GSNs of the second kind and corresponding Cauchy numbers and polynomials. We then derive some further relationships between Bernoulli and Cauchy polynomials. Subsequently, in~Section~\ref{sec:6}, we provide several formulas for the poly-Cauchy polynomials by generalizing previous results concerning ordinary Cauchy equivalents. Furthermore, elaborating on earlier work by Komatsu et~al. \cite{komatsu2,komatsu3,komatsu5}, in~Section~\ref{sec:7} we formulate a type of multiparameter poly-Cauchy polynomial encompassing, in~particular, shifted poly-Cauchy numbers and polynomials with a $q$~parameter. Finally, we outline our~conclusions.

\section{Basic Cauchy Polynomial~Formulas}\label{sec:2}

We begin this section by deriving several formulas and identities connecting the Cauchy numbers and polynomials with the generalized Stirling numbers, the~Stirling numbers of the first kind, and the binomial coefficients. Then, we will obtain specific formulas for $c_{2n}(x)$, $\widehat{c}_{2n}(x)$, $c_{2n+1}(x)$, and~$\widehat{c}_{2n+1}(x)$ in terms of the central factorial numbers and the respective Bernoulli numbers and Euler polynomials. Lastly, we point out a relationship between the Cauchy polynomials and $r$-Whitney~numbers.

\subsection{Cauchy Polynomials and Generalized Stirling~Numbers}

The following proposition expresses Cauchy polynomials of both kinds in terms of the GSN of the first kind. This result constitutes a fundamental property of Cauchy polynomials and will be the basis for many of our subsequent findings. For~this reason, while the result is already known, we now provide alternative proof of it by exploiting the relationship between the GSN of the first kind and the so-called {\it generalized Stirling polynomials of the first kind} introduced by Adamchik in~\cite{adam} (Equation~(18)).
\begin{proposition}\label{prop:1}
For integers $n \geq 0$ and for arbitrary $x$, we have
\begin{equation}\label{th11}
c_n(x) = \sum_{m=0}^n \frac{(-1)^{n-m}}{m+1} \genfrac{[}{]}{0pt}{}{n}{m}_x,
\end{equation}
and
\begin{equation}\label{th12}
\widehat{c}_n(-x) = (-1)^n \sum_{m=0}^n \frac{1}{m+1} \genfrac{[}{]}{0pt}{}{n}{m}_x,
\end{equation}
where $\genfrac{[}{]}{0pt}{}{n}{m}_x$ is the polynomial defined in \eqref{gs1}.
\end{proposition}
\begin{proof}
To prove \eqref{th12}, consider the polynomials (cf.\ \cite{adam}, Equation~(18))
\begin{equation}\label{adam}
P_{n+1,m}(x) = \sum_{l=m}^{n} \binom{l}{m} \genfrac{[}{]}{0pt}{}{n+1}{l+1} (-x)^{l-m}, \quad 0 \leq m \leq n,
\end{equation}
with the generating function~\cite{adam} (p.\ 276)
\begin{equation}\label{th1p1}
\prod_{l=1}^{n} (t-x+l) = \sum_{m=0}^{n} P_{n+1,m}(x) \, t^m.
\end{equation}
Since
\begin{equation*}
\prod_{l=1}^{n} (t-x+l) = n! \binom{t-x+n}{n} = (-1)^n n! \binom{x-t-1}{n},
\end{equation*}
and $\widehat{c}_n(x) = n! \int_{0}^{1} \binom{x - t}{n} dt$, it follows that
\begin{equation*}
\widehat{c}_n(x) = (-1)^n \sum_{m=0}^n P_{n+1,m}(x+1) \int_{0}^{1} t^m dt =
(-1)^n \sum_{m=0}^n \frac{1}{m+1} P_{n+1,m}(x+1).
\end{equation*}
Hence, to~derive \eqref{th12}, it suffices to show that $P_{n+1,m}(x+1) = \genfrac{[}{]}{0pt}{}{n}{m}_{-x}$ or, equivalently, that $P_{n+1,m}(-x) = \genfrac{[}{]}{0pt}{}{n}{m}_{x+1}$. To~show this, note that, from~the definition \eqref{adam} above, we have
\begin{equation*}
P_{n+1,m}(-x) = \sum_{i=m}^{n} \binom{i}{m} \genfrac{[}{]}{0pt}{}{n+1}{i+1} x^{i-m}
 = \sum_{i=0}^{n-m} \binom{i+m}{m} \genfrac{[}{]}{0pt}{}{n+1}{i+m+1} x^{i}.
\end{equation*}
However, by~the identity in~\cite{cere} (Equation~(24)), the~rightmost side of this last equation is equal to $\genfrac{[}{]}{0pt}{}{n}{m}_{x+1}$, and~the proof of \eqref{th12} is complete. Similarly, Formula \eqref{th11} can be deduced from \eqref{th1p1} by replacing $t$ by $-t$ and $x$ by $1-x$.
\end{proof}

The following are a few observations regarding the above proposition.
\begin{itemize}
\item
If $x$ is the non-negative integer $r$, identities \eqref{th11} and \eqref{th12} become
\begin{equation*}
c_n(r) = \sum_{m=0}^n \frac{(-1)^{n-m}}{m+1} \genfrac{[}{]}{0pt}{}{n}{m}_r
\quad \, \text{and}\quad \widehat{c}_n(-r) = (-1)^n \sum_{m=0}^n \frac{1}{m+1} \genfrac{[}{]}{0pt}{}{n}{m}_r,
\end{equation*}
respectively, thus recovering Theorems 1 and 2 of~\cite{mezo}. In~particular, for~$r=0$, we have (\cite{merlini}, pp.~1908, 1910)
\begin{equation*}
c_n = \sum_{m=0}^n \frac{(-1)^{n-m}}{m+1} \genfrac{[}{]}{0pt}{}{n}{m}
\quad \, \text{and}\quad \widehat{c}_n = (-1)^n \sum_{m=0}^n \frac{1}{m+1} \genfrac{[}{]}{0pt}{}{n}{m}.
\end{equation*}

\item
In view of \eqref{th12}, $\widehat{c}_n(x)$ can equally be expressed as
\begin{equation}\label{rem31}
\widehat{c}_n(x) = (-1)^n \sum_{m=0}^n \frac{1}{m+1} \genfrac{[}{]}{0pt}{}{n}{m}_{-x},
\end{equation}
where
\begin{equation}\label{rem32}
\genfrac{[}{]}{0pt}{}{n}{m}_{-x} = \sum_{i=0}^{n-m} (-1)^i \binom{i+m}{m} \genfrac{[}{]}{0pt}{}{n}{i+m} x^i.
\end{equation}
Alternatively, we have
\begin{equation*}
\widehat{c}_n(x) = \sum_{m=0}^n \frac{(-1)^{n-m}}{m+1} \genfrac{[}{]}{0pt}{}{n}{m}_{1-x},
\end{equation*}
from which we obtain (\cite{kargin}, Equation~(16))
\begin{equation}\label{kargin}
\widehat{c}_n = \sum_{m=0}^n \frac{(-1)^{n-m}}{m+1} \genfrac{[}{]}{0pt}{}{n+1}{m+1}.
\end{equation}

\item
Inverting \eqref{th11} and \eqref{th12} yields
\begin{equation}\label{rem4}
\sum_{m=0}^n \genfrac{\{}{\}}{0pt}{}{n}{m}_x c_m(x) = \frac{1}{n+1}
\quad \, \text{and}\quad \sum_{m=0}^n \genfrac{\{}{\}}{0pt}{}{n}{m}_x \widehat{c}_m(-x)
= \frac{(-1)^n}{n+1},
\end{equation}
generalizing Theorem 3 of~\cite{mezo}. For~$x=0$, this gives
\begin{equation*}
\sum_{m=0}^n \genfrac{\{}{\}}{0pt}{}{n}{m} c_m = \frac{1}{n+1}
\quad \, \text{and}\quad \sum_{m=0}^n \genfrac{\{}{\}}{0pt}{}{n}{m} \widehat{c}_m
= \frac{(-1)^n}{n+1},
\end{equation*}
which corresponds to Theorem 2.3 and the second identity in Theorem 2.6 of~\cite{merlini}, respectively.
\end{itemize}

The following corollary provides an explicit formula for the coefficients of $c_n(x)$ and~$\widehat{c}_n(x)$.
\begin{corollary}\label{col:1}
Let $c_n(x) = \sum_{i=0}^n c_{n,i} x^i$ and $\widehat{c}_n(x) = \sum_{i=0}^n \widehat{c}_{n,i} x^i$ be Cauchy polynomials of the first and second kind, respectively. Then, for~$i = 0,1,\ldots,n$, we have
\begin{equation}
c_{n,i} = (-1)^{n+i} \sum_{m=i}^n \frac{(-1)^m}{m-i+1} \binom{m}{i}
\genfrac{[}{]}{0pt}{}{n}{m}, \label{coef1}
\end{equation}
and
\begin{equation}
\widehat{c}_{n,i} = (-1)^{n+i} \sum_{m=i}^n \frac{1}{m-i+1} \binom{m}{i}
\genfrac{[}{]}{0pt}{}{n}{m}. \label{coef2}
\end{equation}
\end{corollary}
\begin{proof}
Substituting \eqref{gs1} into \eqref{th11}, we immediately derive \eqref{coef1}. Similarly, substituting \eqref{rem32} in \eqref{rem31}, we obtain \eqref{coef2}.
\end{proof}

Clearly, the~coefficients $c_{n,0}$ and $\widehat{c}_{n,0}$ reduce to the known Cauchy number formulas $c_{n}$ and $\widehat{c}_{n}$, respectively. Furthermore, the~leading coefficients are given by $c_{n,n} = (-1)^n$ and $\widehat{c}_{n,n} = 1$ for all $n \geq 0$; likewise, $c_{n,n-1} = \frac{(-1)^n}{2} n(n -2)$ and $\widehat{c}_{n,n-1} = -\frac{1}{2}n^2$ for all $n \geq 1$.

\subsection{Cauchy Polynomials and Stirling Numbers of the First~Kind}

In this subsection, we derive an explicit Cauchy polynomial formula in terms of ordinary Stirling numbers of the first kind. Before proving this result, we present the following integration formulas involving power sum polynomials.
\begin{lemma}\label{lm:1}
Let $S_n(x)$ be the polynomial defined in \eqref{defb} and let $y \in \mathbb{R}$. Then, for~any integers $n \geq 1$ and for arbitrary $y$, we have
\begin{equation}\label{lm11}
\int_{0}^{1} S_{n-1}(x+y-1) \text{d}x = \frac{1}{n} \big( y^n - B_n(1) \big),
\end{equation}
and
\begin{equation}\label{lm12}
\int_{0}^{1} S_{n-1}(-x+y-1) \text{d}x = \frac{1}{n} \big( (y-1)^n - B_n(1) \big).
\end{equation}
\end{lemma}
\begin{proof}
The above formulas readily follow from combining \eqref{defb} with the following two properties of the Bernoulli polynomials, namely, the~integration formula~\cite{apostol} (Equation~(29))
\begin{equation*}
\int_{x}^{y} B_n(t) dt = \frac{1}{n+1} \big( B_{n+1}(y) - B_{n+1}(x) \big), \quad n \geq 0,
\end{equation*}
and the difference equation $B_{n}(x+1) - B_{n}(x) = n x^{n-1}$, $n \geq 1$ \cite{apostol} (Equation~(14)).
\end{proof}

Provided with \eqref{lm11}, we now prove the following theorem.
\begin{theorem}\label{th:1}
For integers $n \geq 1$ and for arbitrary $x$, we have
\begin{equation}\label{exp1}
c_n(x) = c_n + (-1)^n n \sum_{m=1}^n \frac{1}{m} \genfrac{[}{]}{0pt}{}{n-1}{m-1} x^m,
\end{equation}
and
\begin{equation}\label{exp2}
\widehat{c}_n(x) = c_n + (-1)^n n \sum_{m=1}^n \frac{(-1)^{m}}{m} \genfrac{[}{]}{0pt}{}{n-1}{m-1}
(x-1)^m,
\end{equation}
where the Cauchy numbers of the first kind $c_n$ can be expressed in terms of Bernoulli numbers as
\begin{equation}
c_n = \delta_{n,1} + (-1)^{n+1} n \sum_{m=1}^n \genfrac{[}{]}{0pt}{}{n-1}{m-1} \frac{B_m}{m},
\quad (n \geq 1). \label{exp3}
\end{equation}
\end{theorem}
\begin{proof}
We start with the relation (cf.\ \cite{cere2}, Equation~(7))
\begin{equation*}\label{ind1}
n!\binom{x+1}{n+1} = \delta_{n,0} + \sum_{m=0}^n (-1)^{n-m} \genfrac{[}{]}{0pt}{}{n}{m} S_m(x), \quad n \geq 0.
\end{equation*}
{By} setting $n \to n-1$, $x \to x -s -1$, and~integrating both sides over $x$ from $0$ to $1$, we~obtain
\begin{equation}\label{poly3}
\frac{c_{n}(s)}{n} = \delta_{n,1} + \sum_{m=1}^n (-1)^{n-m}
\genfrac{[}{]}{0pt}{}{n-1}{m-1} \int_{0}^{1} S_{m-1}(x-s-1) dx,
\end{equation}
which holds for every $n \geq 1$. On~the other hand, applying \eqref{lm11} for $y = -s$ we find
\begin{equation*}
\int_{0}^{1} S_{m-1}(x-s-1) dx = \frac{1}{m}\big( (-s)^{m} - B_{m}(1) \big).
\end{equation*}
{Thus,} it follows from \eqref{poly3} that
\begin{equation*}
\frac{c_{n}(s)}{n} = \delta_{n,1} + \sum_{m=1}^n \frac{(-1)^{n+1}}{m} \genfrac{[}{]}{0pt}{}{n-1}{m-1} B_{m}
+ \sum_{m=1}^n \frac{(-1)^{n}}{m} \genfrac{[}{]}{0pt}{}{n-1}{m-1} s^{m},
\end{equation*}
where $B_m(1) = (-1)^m B_m$ for all $m \geq 0$ (\cite{mikkawy}, Equation~(8)). Next, renaming $s$ to $x$, we obtain
\begin{equation*}
c_{n}(x) = \delta_{n,1} + (-1)^{n+1} n \sum_{m=1}^n \genfrac{[}{]}{0pt}{}{n-1}{m-1} \frac{B_{m}}{m}
+ (-1)^n n \sum_{m=1}^n \frac{1}{m} \genfrac{[}{]}{0pt}{}{n-1}{m-1} x^{m}.
\end{equation*}
{Clearly,} as~$c_n = c_n(0)$, \eqref{exp3} is true, and~the proof of \eqref{exp1} is complete. Equation~\eqref{exp2} is obtained by letting $x \to 1-x$ in \eqref{exp1}.
\end{proof}

\begin{remark}
In view of \eqref{exp1}, the~coefficients $c_{n,i}$ in the expansion $c_n(x) = c_n + \sum_{i=1}^n c_{n,i} x^i$ are given by
\begin{equation}\label{exp4}
c_{n,i} = (-1)^n \frac{n}{i} \genfrac{[}{]}{0pt}{}{n-1}{i-1}, \quad 1 \leq i \leq n.
\end{equation}
{Therefore,} equating \eqref{coef1} and \eqref{exp4} reveals the identity
\begin{equation*}
\sum_{m=i}^n \frac{(-1)^{m-i}}{m-i+1} \binom{m}{i} \genfrac{[}{]}{0pt}{}{n}{m}
= \frac{n}{i} \genfrac{[}{]}{0pt}{}{n-1}{i-1}, \quad 1 \leq i \leq n.
\end{equation*}
{In} particular, for~$i=1$, we find the well-known identity $\sum_{m=1}^n (-1)^{m} \genfrac{[}{]}{0pt}{}{n}{m} = 0$,  ($n \geq 2$).
\end{remark}

Let us observe that, by~integrating \eqref{exp1} and \eqref{exp2} from $0$ to $1$, we immediately obtain
\begin{equation*}
\int_0^1 c_n(x) dx = \int_0^1 \widehat{c}_n(x) dx = c_n + (-1)^n n \sum_{m=1}^n \frac{1}{m(m+1)} \genfrac{[}{]}{0pt}{}{n-1}{m-1}.
\end{equation*}
{On the} other hand, according to Theorem 1.1 in~\cite{qi} (see also Equation~(17) in~\cite{kargin}), Cauchy numbers of the first kind can be expressed as
\begin{equation*}
c_n = (-1)^{n+1} \sum_{m=1}^n \frac{1}{m(m+1)} \genfrac{[}{]}{0pt}{}{n-1}{m-1},  \quad n \geq 1.
\end{equation*}
{Hence,} combining the last two equations gives us the integration formula
\begin{equation}\label{int1}
\int_0^1 c_n(x) dx = \int_0^1 \widehat{c}_n(x) dx = (1 -n) c_n,  \quad n \geq 0.
\end{equation}

\subsection{Cauchy Polynomials and Binomial~Coefficients}

In order to prove Theorems \ref{th:2} and \ref{th:3} in this subsection, we make use of a number of results recently obtained by Chen and Guo~\cite{chen}. The~following theorem expresses Cauchy polynomials in terms of the corresponding Cauchy numbers and binomial coefficients depending on the indeterminate $x$.
\begin{theorem}\label{th:2}
{For} integers $n \geq 0$ and for arbitrary $x$, we have
\begin{equation}\label{chen1}
c_n(x) = (-1)^n n! \sum_{m=0}^n \frac{\widehat{c}_m}{m!} \binom{x+n-1}{n-m},
\end{equation}
and
\begin{equation}\label{chen2}
\widehat{c}_n(x) = n! \sum_{m=0}^n \frac{(-1)^m \, c_m}{m!} \binom{x-m}{n-m}.
\end{equation}
\end{theorem}
\begin{proof}
We prove only \eqref{chen1}, but~the proof of \eqref{chen2} is similar. For~this, consider the number sequence $\mathcal{A}_x(n,m)$ defined by
\begin{equation}\label{guodf}
\mathcal{A}_{x}(n,m) = \frac{n!}{m!}\binom{x +n-1}{n-m}, \quad 0\leq m \leq n,
\end{equation}
for an indeterminate $x$. According to~\cite{chen} (top of p.\ 2), this sequence fulfills the summation~formula
\begin{equation*}
\sum_{m=k}^{n} \binom{m}{k} \genfrac{[}{]}{0pt}{}{n}{m} x^{m-k}
= \sum_{m=k}^{n} (-1)^{m-k} \genfrac{[}{]}{0pt}{}{m}{k} \mathcal{A}_{x}(n,m).
\end{equation*}
{Also} note that left-hand side is equal to $\genfrac{[}{]}{0pt}{}{n}{k}_{x}$ (cf.\ Equation~\eqref{gs1}). Thus, using \eqref{guodf}, we have
\begin{equation}\label{alpha1}
\genfrac{[}{]}{0pt}{}{n}{k}_{x} = n! \sum_{m=k}^{n} \frac{(-1)^{m-k}}{m!}
\genfrac{[}{]}{0pt}{}{m}{k} \binom{x +n-1}{n-m}.
\end{equation}
{Substituting} this expression for $\genfrac{[}{]}{0pt}{}{n}{k}_{x}$ into \eqref{th11} yields
\begin{align*}
c_n(x) & = (-1)^n n! \sum_{k=0}^n \frac{(-1)^{k}}{k+1} \sum_{m=k}^{n} \frac{(-1)^{m-k}}{m!}
\genfrac{[}{]}{0pt}{}{m}{k} \binom{x+n-1}{n-m} \\[1mm]
& = (-1)^n n! \sum_{m=0}^n \frac{(-1)^m}{m!} \binom{x+n-1}{n-m} \sum_{k=0}^m \frac{1}{k+1}
\genfrac{[}{]}{0pt}{}{m}{k},
\end{align*}
which can be put in the form of Equation~\eqref{chen1} by noticing that $\widehat{c}_m = (-1)^m \sum_{k=0}^m \frac{1}{k+1}\genfrac{[}{]}{0pt}{}{m}{k}$.
\end{proof}

\begin{remark}
By taking $x=0$ or $x=1$ in \eqref{chen1} and \eqref{chen2} we get
\begin{align*}
\widehat{c}_n & = (-1)^n n! \sum_{m=0}^n \binom{n}{m} \frac{\widehat{c}_m}{m!}, \quad n \geq 0, \\
c_n & = (-1)^n n! \sum_{m=1}^n \binom{n-1}{m-1} \frac{\widehat{c}_m}{m!}, \quad n \geq 1, \\
\widehat{c}_n & = (-1)^n n! \sum_{m=1}^n \binom{n-1}{m-1} \frac{c_m}{m!}, \quad n \geq 1, \\
c_n & = (-1)^n n! \sum_{m=2}^n \binom{n-2}{m-2} \frac{c_m}{m!}, \quad n \geq 2.
\end{align*}
{The first} three identities above are already known (see, e.g.,~\cite{merlini} (Theorem 2.7) and generalized to poly-Cauchy numbers in~\cite{komatsu} (Theorem 7)); the last one, however, seems to be new.
\end{remark}

It should be noted that the symmetric transformation formula in~\cite{chen} (Proposition 10) entails the following alternative formula for the GSN of the first kind:
\begin{equation}\label{alpha2}
\genfrac{[}{]}{0pt}{}{n}{k}_{x} = \sum_{m=k}^{n} \binom{n}{m}
\genfrac{[}{]}{0pt}{}{m}{k} \binom{x +n-m-1}{n-m}.
\end{equation}
{Consequently,} employing \eqref{alpha2} instead of \eqref{alpha1} in the proof of Theorem \ref{th:2} produces the following variant of Equations~\eqref{chen1} and \eqref{chen2}:
\begin{equation}\label{symm1}
c_n(x) = (-1)^n n! \sum_{m=0}^n \frac{(-1)^m c_m}{m!} \binom{x+n-m-1}{n-m},
\end{equation}
and
\begin{equation}\label{symm2}
\widehat{c}_n(x) = n! \sum_{m=0}^n \frac{\widehat{c}_m}{m!} \binom{x}{n-m},
\end{equation}
respectively.

We now make the following~observations:
\begin{itemize}
\item
Putting $x=1$ in \eqref{symm1} and \eqref{symm2} gives (\cite{merlini}, Equations~(2.1) and (2.2))
\begin{equation*}
\widehat{c}_n = (-1)^n n! \sum_{m=0}^n \frac{(-1)^m c_m}{m!} \!\quad (n \geq 0) \quad\text{and} \!\quad
c_n = \widehat{c}_n + n \widehat{c}_{n-1} \quad (n \geq 1).
\end{equation*}

\item
Integrating \eqref{symm2} from $0$ to $1$ yields
\begin{equation*}
\int_0^1 \widehat{c}_n(x) dx =  n! \sum_{m=0}^n \frac{\widehat{c}_m}{m!} \int_0^1 \binom{x}{n-m} dx
=  \sum_{m=0}^n \binom{n}{m} \widehat{c}_m c_{n-m}.
\end{equation*}
So, recalling \eqref{int1}, it follows that
\begin{equation*}
\sum_{m=0}^n \binom{n}{m} \widehat{c}_m c_{n-m} = \sum_{m=0}^n \binom{n}{m} c_m \widehat{c}_{n-m} = (1-n)c_n,
\end{equation*}
in accordance with~\cite{merlini} (Equation~(2.3)) (see also~\cite{komatsu7}, Equation~(101)).

\item
Furthermore, it can be shown that
\begin{align*}
c_n(2) & = (-1)^n n! \sum_{m=0}^n \frac{(-1)^m \widehat{c}_m}{m!}, \\
& = (-1)^n n! \sum_{m=0}^n \binom{n+1}{m+1} \frac{\widehat{c}_m}{m!}, \\
& = (-1)^n n! \sum_{m=0}^n \binom{n}{m} \frac{c_m}{m!}.
\end{align*}
\end{itemize}

The next theorem gives a new recurrence formula for Cauchy polynomials.
\begin{theorem}\label{th:3}
For integers $n \geq 0$ and for arbitrary $x$, we have
\begin{equation}
c_{n+1}(x) = -(n+x)c_n(x) +(-1)^{n+1} n! \sum_{m=0}^n \frac{\widehat{c}_{m+1}}{m!}
\binom{x+n}{n-m}, \label{chen3}
\end{equation}
and
\begin{equation}
\widehat{c}_{n+1}(x) = (x-n)\widehat{c}_n(x) - n! \sum_{m=0}^n \frac{(-1)^m \, c_{m+1}}{m!}
\binom{x-m-1}{n-m}. \label{chen4}
\end{equation}
\end{theorem}
\begin{proof}
To obtain \eqref{chen3}, write $c_{n+1}(x)$ in the form
\begin{equation*}
c_{n+1}(x) = \sum_{k=0}^{n+1} \frac{(-1)^{n+1-k}}{k+1} \sum_{m=k}^{n+1} (-1)^{m-k}
\genfrac{[}{]}{0pt}{}{m}{k} \mathcal{A}_{x}(n+1,m),
\end{equation*}
and then use the relation~\cite{chen} (Proposition 1)
\begin{equation}\label{guo}
\mathcal{A}_{x}(n+1,m) = \mathcal{A}_{x}(n,m-1) + (n+m+x) \mathcal{A}_{x}(n,m),
\end{equation}
together with \eqref{chen1}. Similarly, to~obtain \eqref{chen4}, write $\widehat{c}_{n+1}(x)$ in the form
\begin{equation*}
\widehat{c}_{n+1}(x) = (-1)^{n+1} \sum_{k=0}^{n+1} \frac{1}{k+1} \sum_{m=k}^{n+1} (-1)^{m-k}
\genfrac{[}{]}{0pt}{}{m}{k} \mathcal{A}_{-x}(n+1,m),
\end{equation*}
and then use the relation \eqref{guo} (with $x$ replaced by $-x$) together with \eqref{chen2}.
\end{proof}

In a similar manner, using the relation~\cite{chen} (Proposition 3)
\begin{equation*}
m\mathcal{A}_{x}(n,m) = m n\mathcal{A}_{x}(n-1,m) + n \mathcal{A}_{x}(n-1,m-1),
\end{equation*}
it can be shown that, for~$n \geq 1$,
\begin{equation*}
c_{n}(x) = -n c_{n-1}(x) + (-1)^{n} n! \sum_{m=0}^{n} \frac{\widehat{c}_{m}}{m!}
\binom{x+n-2}{n-m},
\end{equation*}
and
\begin{equation*}
\widehat{c}_{n}(x) = -n \widehat{c}_{n-1}(x) + n! \sum_{m=0}^{n} \frac{(-1)^m \,c_{m}}{m!}
\binom{x+1-m}{n-m}.
\end{equation*}

\begin{remark}
Combining the last two equations with Theorem \ref{th:2} yields the difference equations ($n \geq 1$)
\begin{equation}\label{diff1}
c_{n}(x+1) - c_{n}(x) = -n c_{n-1}(x+1),
\end{equation}
and
\begin{equation}\label{diff2}
\widehat{c}_{n}(x+1) - \widehat{c}_{n}(x) = n \widehat{c}_{n-1}(x).
\end{equation}
{In addition,} by~setting $x \to -x$ in either of the two equations above, we immediately obtain
\begin{equation*}
c_n(x) = \widehat{c}_n(-x) + n \widehat{c}_{n-1}(-x),
\end{equation*}
which reduces to the known relation $c_n = \widehat{c}_n + n \widehat{c}_{n-1}$ when $x=0$.
\end{remark}

To conclude this subsection, it is pertinent to note that the Formulas \eqref{chen1} and \eqref{chen2} in Theorem~\ref{th:2} can be written in the form:
\begin{equation}\label{seq1}
c_n(x) = (-1)^n \sum_{m=0}^n \widehat{c}_m \mathcal{A}_x(n,m),
\end{equation}
and
\begin{equation}\label{seq2}
\widehat{c}_n(-x) = (-1)^n \sum_{m=0}^n c_m \mathcal{A}_x(n,m),
\end{equation}
expressing $c_n(x)$ and $\widehat{c}_n(-x)$ in the basis $\{ \mathcal{A}_x(n,m) \}_{m=0}^n$ of the space of polynomials in $x$ of degree not greater than $n$, defined by \eqref{guodf}. We point out that \eqref{seq1} and \eqref{seq2} are equivalent, respectively, to~Formulas (10) and (9) in Proposition 20 of~\cite{chen}. Moreover, differentiating \eqref{seq1} and \eqref{seq2} with respect to $x$ gives rise to the formulas in Proposition 21 of~\cite{chen}.

\subsection{Cauchy Polynomials, Central Factorial Numbers and Euler~Polynomials}

The following theorem provides a formula for $c_{2n}(x)$ and $\widehat{c}_{2n}(x)$ in terms of central factorial numbers and Bernoulli numbers, as~well as a formula for $c_{2n+1}(x)$ and $\widehat{c}_{2n+1}(x)$ in terms of central factorial numbers and Euler polynomials. A~through account of the central factorial numbers can be found in~\cite{butzer}. For~Bernoulli and Euler polynomials, see, e.g.,~\cite{euler1}.
\begin{theorem}\label{th:4}
For integers $n \geq 1$ and for arbitrary $x$, we have
\begin{align*}
c_{2n}(x) & = n \sum_{m=1}^n \frac{u(n,m)}{m} \big( (x+n-1)^{2m} - B_{2m} \big), \\
\widehat{c}_{2n}(x) & = n \sum_{m=1}^n \frac{u(n,m)}{m} \big( (x-n)^{2m} - B_{2m} \big), \\
c_{2n+1}(x) & = -(2n+1) \sum_{m=1}^n \frac{u(n,m)}{2m+1} E_{2m+1}(x+n), \\
\widehat{c}_{2n+1}(x) & = (2n+1) \sum_{m=1}^n \frac{u(n,m)}{2m+1} E_{2m+1}(x -n),
\end{align*}
where $u(n,m) = t(2n,2m)$ are the (signed) central factorial numbers with even indices of the first kind and $E_m(x)$ are the Euler polynomials.
\end{theorem}
\begin{proof}
We prove only the third identity. For~this, we start with the relation (cf.\ \cite{cere3}, p.\ 7)
\begin{equation*}
(2n)! \binom{x+n+1}{2n+1} = \sum_{m=1}^n u(n,m) 2^{2m+1} S_{2m} \Big(\frac{x}{2} \Big), \quad n \geq 1.
\end{equation*}
{Setting} $x \to x -s -n-1$ and integrating over $x$ from $0$ to $1$ on both sides gives
\begin{equation}\label{cf1}
\frac{c_{2n+1}(s)}{2n+1} = \sum_{m=1}^n \frac{u(n,m)}{2m+1} 2^{2m+2}
\int_{\frac{-s-n+1}{2}}^{\frac{-s-n+2}{2}} B_{2m+1}(x) dx,
\end{equation}
where, for~$m \geq 1$,
\begin{align*}
\int_{0}^{1} S_{2m}\Big(\frac{1}{2}(x -s-n-1) \Big) dx
& = \frac{1}{2m+1} \int_{0}^{1} B_{2m+1} \Big(\frac{1}{2}(x -s-n+1) \Big) dx \\
& = \frac{2}{2m+1} \int_{\frac{-s-n+1}{2}}^{\frac{-s-n+2}{2}} B_{2m+1}(x) dx.
\end{align*}
{Next,} employing the integration formula~\cite{moll} (Equation~(2.5))
\begin{equation*}
\int_{a}^{a +\frac{1}{2}} B_n(x) dx = \frac{1}{2^{n+1}} E_n(2a),  \quad n \geq 0,
\end{equation*}
we obtain
\begin{equation*}
\int_{\frac{-s-n+1}{2}}^{\frac{-s-n+2}{2}} B_{2m+1}(x) dx = \frac{E_{2m+1}(1-s-n)}{2^{2m+2}}
= -\frac{E_{2m+1}(s+n)}{2^{2m+2}}.
\end{equation*}
{Thus,} the~stated identity follows upon replacing the integral in \eqref{cf1} by $-\frac{E_{2m+1}(s+n)}{2^{2m+2}}$ and writing $x$ instead of $s$.
\end{proof}

\subsection{Cauchy Polynomials and $r$-Whitney~Numbers}

$r$-Whitney numbers of the first and second kind, $w_{m,r}(n,l)$ and $W_{m,r}(n,l)$, $0 \leq l \leq n$, were introduced by Mez\H{o} \cite{mezo2} as a new class of numbers generalizing Whitney and $r$-Stirling numbers. They can be defined as connection constants in the polynomial identities~\cite{cheon3} (Equations~(1) and (2))
\begin{equation*}
m^n (x)_n = \sum_{l=0}^n (-1)^{n-l} w_{m,r}(n,l) (mx +r)^l,
\end{equation*}
and
\begin{equation*}
(mx +r)^n = \sum_{l=0}^n m^l W_{m,r}(n,l) (x)_l,
\end{equation*}
where $(x)_n$ denotes the falling factorial $(x)_n = x(x-1)\ldots (x-n+1)$ for $n \geq 1$ with $(x)_0 =1$.

Alternatively, $r$-Whitney numbers can be expressed in terms of ordinary Stirling numbers as follows (\cite{cheon3} (pp.~2343--2344)):
\begin{equation*}
w_{m,r}(n,l) = \sum_{j=l}^n \binom{j}{l} m^{n-j} r^{j-l} \genfrac{[}{]}{0pt}{}{n}{j}
\quad\, \text{and}\quad W_{m,r}(n,l) = \sum_{j=l}^n \binom{n}{j} m^{j-l} r^{n-j} \genfrac{\{}{\}}{0pt}{}{j}{l},
\end{equation*}
or, equivalently,
\begin{equation*}
w_{m,r}(n,l) = m^{n-l} \genfrac{[}{]}{0pt}{}{n}{l}_{\frac{r}{m}} \, \quad \text{and} \quad
W_{m,r}(n,l) = m^{n-l} \genfrac{\{}{\}}{0pt}{}{n}{l}_{\frac{r}{m}},
\end{equation*}
provided that $m \neq 0$.

Therefore, assuming that $r$ and $m$ are arbitrary integers with $m \neq 0$, we may combine the above formula for $w_{m,r}(n,l)$ with Proposition \ref{prop:1} to obtain
\begin{equation}\label{whit1}
c_n \left( \frac{r}{m} \right) = \sum_{l=0}^n \frac{(-1)^{n-l}}{l+1} \, \frac{w_{m,r}(n,l)}{m^{n-l}},
\end{equation}
and
\begin{equation}\label{whit2}
\widehat{c}_n \left( -\frac{r}{m} \right) = (-1)^n  \sum_{l=0}^n \frac{1}{l+1} \, \frac{w_{m,r}(n,l)}{m^{n-l}},
\end{equation}
which allow us to evaluate $c_n(x)$ [resp.\ $\widehat{c}_n(x)$] at the rational number $\frac{r}{m}$ [resp.\ $-\frac{r}{m}$], once we known the values of $w_{m,r}(n,l)$, $l =0,1,\ldots,n$.

We note that \eqref{whit1} and \eqref{whit2} can be reversed to
\begin{equation*}
\sum_{l=0}^n m^l W_{m,r}(n,l) \, c_l \left( \frac{r}{m} \right) = \frac{m^n}{n+1},
\end{equation*}
and
\begin{equation*}
\sum_{l=0}^n m^l W_{m,r}(n,l) \, \widehat{c}_l \left( -\frac{r}{m} \right) = (-1)^n \frac{m^n}{n+1},
\end{equation*}
respectively.

\begin{remark}
In~\cite{shiha}, Shiha obtained explicit formulas for computing Cauchy polynomials with a $q$ parameter in terms of $r$-Whitney numbers of the first kind.
\end{remark}

\section{Derivative Formulas for Cauchy~Polynomials}\label{sec:3}

In this section, we establish some noteworthy formulas for higher-order derivatives of Cauchy polynomials. To this end, we first note that the generalized Bernoulli polynomials defined by \eqref{gfgen} can be expressed as derivatives of binomial coefficients as follows (see, e.g.,~\cite{gould} (Equation~(13.2))):
\begin{equation*}
B_{\nu}^{(m+1)}(x+1) = \nu! \, \frac{d^{m-\nu}}{d x^{m-\nu}} \binom{x}{m},   \quad m \geq \nu.
\end{equation*}
{Combining} this expression with \eqref{rep1} results in
\begin{equation}\label{gs1ber}
\genfrac{[}{]}{0pt}{}{m}{i}_x = (-1)^{m-i} \binom{m}{i} B_{m-i}^{(m+1)}(1-x).
\end{equation}
{Hence,} by~using \eqref{gs1ber} in \eqref{th11} and \eqref{th12}, we obtain the following formulas for Cauchy polynomials in terms of generalized Bernoulli polynomials:
\begin{equation}
c_n(x) = \sum_{m=0}^n \binom{n}{m} \frac{B_{m}^{(n+1)}(1-x)}{n+1-m}, \label{gbp1}
\end{equation}
and
\begin{equation}
\widehat{c}_n(x) = \sum_{m=0}^n (-1)^{n-m} \binom{n}{m} \frac{B_{m}^{(n+1)}(x+1)}{n+1-m}. \label{gbp2}
\end{equation}

On the other hand, the~generalized Bernoulli polynomials satisfy, as~an Appell sequence, the~following well-known rule for the derivatives with respect to $x$: $\frac{d^i B_m^{(n)}(x)}{dx^i} = i! \binom{m}{i} B_{m-i}^{(n)}(x)$ (cf.\ \cite{kellner}, Equation~(1.1)). Then, from~\eqref{gbp1} and \eqref{gbp2}, we quickly obtain the following formulas for the $i$-th derivative of Cauchy polynomials of both kinds.
\begin{theorem}\label{th:5}
For integers $i \geq 0$, the~$i$-th derivative of Cauchy polynomials can be expressed by
\begin{equation*}
\frac{d^i c_n(x)}{d x^i} = (-1)^i i! \sum_{m=i}^n \binom{n}{m} \binom{m}{i}
\frac{B_{m-i}^{(n+1)}(1-x)}{n+1-m},
\end{equation*}
and
\begin{equation*}
\frac{d^i \widehat{c}_n(x)}{d x^i} = i! \sum_{m=i}^n (-1)^{n-m} \binom{n}{m} \binom{m}{i}
\frac{B_{m-i}^{(n+1)}(x+1)}{n+1-m},
\end{equation*}
\end{theorem}

The next theorem provides an alternative formula for the $i$-th derivative of Cauchy polynomials in terms of the respective Cauchy numbers and the GSN of the first kind.
\begin{theorem}\label{th:6}
For integers $i \geq 0$, the~$i$-th derivative of Cauchy polynomials can be expressed by
\begin{equation}
\frac{d^i c_n(x)}{d x^i} = i! \sum_{m=i}^n (-1)^m c_{n-m} \binom{n}{m} \genfrac{[}{]}{0pt}{}{m}{i}_x,
\label{der1}
\end{equation}
and
\begin{equation}
\frac{d^i \widehat{c}_n(x)}{d x^i} = (-1)^i i! \sum_{m=i}^n (-1)^m \, \widehat{c}_{n-m} \binom{n}{m}
\genfrac{[}{]}{0pt}{}{m}{i}_{-x}. \label{der2}
\end{equation}
\end{theorem}
\begin{proof}
Equation~\eqref{der1} follows straightforwardly by putting \eqref{symm1} in the form
\begin{equation*}
c_n(x) = n! \sum_{m=0}^n \frac{(-1)^m c_{n-m}}{(n-m)!} \binom{x+m-1}{m},
\end{equation*}
and using the relation (cf.\ Equation~\eqref{rep1})
\begin{equation*}
\frac{d^i}{dx^i} \binom{x+m-1}{m} = \frac{i!}{m!} \genfrac{[}{]}{0pt}{}{m}{i}_x.
\end{equation*}
{Similarly,} Equation~\eqref{der2} can be obtained by putting \eqref{symm2} in the form
\begin{equation*}
\widehat{c}_n(x) = n! \sum_{m=0}^n \frac{\widehat{c}_{n-m}}{(n-m)!} \binom{x}{m},
\end{equation*}
and using the relation (cf.\ Equation~\eqref{rep1})
\begin{equation*}
\frac{d^i}{dx^i} \binom{x}{m} = (-1)^{m-i} \frac{i!}{m!} \genfrac{[}{]}{0pt}{}{m}{i}_{-x}.\qedhere
\end{equation*}
\end{proof}

\begin{remark}
By virtue of \eqref{gs1ber}, the~above Formulas \eqref{der1} and \eqref{der2} can be rewritten as
\begin{equation*}
\frac{d^i c_n(x)}{d x^i} =  (-1)^i i! \sum_{m=i}^n c_{n-m} \binom{n}{m} \binom{m}{i}
B_{m-i}^{(m+1)}(1-x),
\end{equation*}
and
\begin{equation*}
\frac{d^i \widehat{c}_n(x)}{d x^i} = i! \sum_{m=i}^n \widehat{c}_{n-m} \binom{n}{m} \binom{m}{i}
B_{m-i}^{(m+1)}(x+1).
\end{equation*}
\end{remark}

We can also obtain derivative formulas for Cauchy polynomials by directly differentiating \eqref{exp1} and \eqref{exp2}. So, for~$i \geq 1$, this yields
\begin{equation}\label{der12}
\frac{d^i c_n(x)}{d x^i} = (-1)^n n (i-1)! \sum_{m=i}^n \binom{m-1}{i-1} \genfrac{[}{]}{0pt}{}{n-1}{m-1} x^{m-i} = (-1)^n n (i-1)! \genfrac{[}{]}{0pt}{}{n-1}{i-1}_x,
\end{equation}
and
\begin{equation}\label{der22}
\frac{d^i \widehat{c}_n(x)}{d x^i} = (-1)^n n (i-1)! \sum_{m=i}^n (-1)^m \binom{m-1}{i-1}
\genfrac{[}{]}{0pt}{}{n-1}{m-1} (x-1)^{m-i} = (-1)^{n+i} n (i-1)! \genfrac{[}{]}{0pt}{}{n-1}{i-1}_{1-x},
\end{equation}
respectively.

Furthermore, identifying the right-hand sides of \eqref{der1} and \eqref{der12}, we obtain
\begin{equation}\label{gder1}
\sum_{m=i}^n (-1)^{n-m} \, c_{n-m} \binom{n}{m} \genfrac{[}{]}{0pt}{}{m}{i}_x
= \frac{n}{i} \genfrac{[}{]}{0pt}{}{n-1}{i-1}_x, \quad i \geq 1.
\end{equation}
{In particular,} for~$(i,x) = (1,0), (2,0), (1,1)$, \eqref{gder1} leads to the following identities:
\begin{gather}
\frac{c_n}{n!} = \sum_{m=0}^{n-1} \frac{c_m}{m!} \, \frac{(-1)^{n+1-m}}{n+1-m},
\quad n \geq 1, \label{merlin} \\
\sum_{m=0}^{n} \frac{(-1)^{n-m}}{m+1} \, \frac{c_{n-m}}{(n-m)!} H_{m} = \frac{1}{2n},
\quad n \geq 1, \notag \\
\sum_{m=0}^{n} (-1)^{n-m} \, \frac{c_{n-m}}{(n-m)!} H_{m+1} = 1, \quad n \geq 0, \label{zhao}
\end{gather}
where \eqref{merlin} is Theorem 2.2 of~\cite{merlini} and \eqref{zhao} is the first identity in Theorem 5.1 of~\cite{zhao} (but note the small error there).

Likewise, by~equating the right-hand sides of \eqref{der2} and \eqref{der22}, and~making $x \to -x$, we~obtain
\begin{equation}\label{gder2}
\sum_{m=i}^n (-1)^{n-m} \, \widehat{c}_{n-m} \binom{n}{m} \genfrac{[}{]}{0pt}{}{m}{i}_{x}
= \frac{n}{i} \genfrac{[}{]}{0pt}{}{n-1}{i-1}_{x+1}, \quad i \geq 1.
\end{equation}
{In particular,} for~$(i,x) = (1,0), (2,0), (1,1)$, \eqref{gder2} leads to the following identities:
\begin{gather*}
\frac{\widehat{c}_n}{n!} = (-1)^n + \sum_{m=0}^{n-1} \frac{\widehat{c}_m}{m!} \,
\frac{(-1)^{n+1-m}}{n+1-m}, \quad n \geq 1, \notag \\
\sum_{m=0}^{n} \frac{(-1)^{n-m}}{m+1} \, \frac{\widehat{c}_{n-m}}{(n-m)!} H_{m} =
\frac{1}{2} H_n, \quad n \geq 0, \notag \\
\sum_{m=0}^{n} (-1)^{n-m} \, \frac{\widehat{c}_{n-m}}{(n-m)!} H_{m+1} = n+1,
\quad n \geq 0,
\end{gather*}
where the last equation is the second identity in Theorem 5.1 of~\cite{zhao}.

\section{A Connection with Hyperharmonic~Polynomials}\label{sec:4}

The hyperharmonic polynomials $H_n^{(x)}$ can be defined by means of the generating function~\cite{benjamin} (p.~2)
\begin{equation}\label{gfhp}
\sum_{n=0}^{\infty} H_n^{(x)} t^n = - \frac{\ln (1-t)}{(1-t)^x}  \quad (|t| < 1),
\end{equation}
with alternative representation~\cite{benjamin} (Theorem 1)
\begin{equation}\label{rep2}
H_n^{(x)} = \sum_{t=1}^n \binom{x+n-t-1}{n-t} \frac{1}{t}, \quad n \geq 1,
\end{equation}
and $H_0^{(x)} =0$. When $x$ is the non-negative integer $r$, $H_n^{(r)}$ is the so-called $n$-th hyperharmonic number of order $r$ \cite{conway} (p.\ 258), with~$H_n^{(0)} = \frac{1}{n}$ for $n \geq 1$. In~particular, $H_n^{(1)}$ is the $n$-th harmonic number $H_n = 1 +\frac{1}{2} + \frac{1}{3} + \cdots + \frac{1}{n}$. The~first few hyperharmonic polynomials are as follows:
\begin{align*}
H_1^{(x)} & = 1, \quad  H_2^{(x)} = \frac{1}{2} +x, \quad  H_3^{(x)} = \frac{1}{3}+ x + \frac{1}{2}x^2,
\quad  H_4^{(x)} = \frac{1}{4}+ \frac{11}{12}x +\frac{3}{4}x^2 + \frac{1}{6} x^3, \\
H_5^{(x)} & = \frac{1}{5} +\frac{5}{6}x + \frac{7}{8}x^2 + \frac{1}{3}x^3 + \frac{1}{24}x^4, \quad H_6^{(x)} = \frac{1}{6}
+\frac{137}{180}x +\frac{15}{16}x^2 +\frac{17}{36} x^3 +\frac{5}{48}x^4 +\frac{1}{120}x^5, \\
H_7^{(x)} & = \frac{1}{7} + \frac{7}{10}x +\frac{29}{30}x^2 +\frac{7}{12}x^3 + \frac{25}{144}x^4 +\frac{1}{40}x^5
+ \frac{1}{720}x^6.
\end{align*}

From the generating functions in \eqref{gf1}, \eqref{gf2} and~\eqref{gfhp}, we can deduce the following relations between Cauchy and hyperharmonic polynomials.
\begin{proposition}\label{prop:2}
For integers $n \geq 0$ and for arbitrary $x$ and $y$, we have
\begin{equation}\label{hyp1}
\sum_{m=0}^n (-1)^{m} \frac{c_m(x)}{m!} H_{n+1-m}^{(y)} = \binom{x+y+n-1}{n},
\end{equation}
and
\begin{equation}\label{hyp2}
\sum_{m=0}^n (-1)^{m} \frac{\widehat{c}_{m}(x)}{m!} H_{n+1-m}^{(y)} = \binom{-x+y+n}{n},
\end{equation}
where $H_n^{(x)}$ are the hyperharmonic polynomials defined by \eqref{gfhp}.
\end{proposition}
\begin{proof}
Equation~\eqref{hyp1} follows from the relation
\begin{equation*}
\left( \sum_{n=0}^{\infty} \frac{c_n(x)}{n!} t^n \right) \left( \sum_{n=0}^{\infty} (-1)^n
H_{n+1}^{(y)} t^n \right) = \frac{1}{(1+t)^{x+y}},
\end{equation*}
after comparing the coefficients of $t^n$ on both sides. On~the other hand, Equation~\eqref{hyp2} is obtained by setting $x \to 1-x$ in \eqref{hyp1}.
\end{proof}

\begin{remark}
The relation in Equation~\eqref{hyp1} was first discovered (although expressed in a slightly different way) by Komatsu and Mez\H{o} (\cite{mezo}, Proposition 2).
\end{remark}

We present the following observations with respect to Proposition \ref{prop:2}:
\begin{itemize}
\item
From \eqref{hyp1} and \eqref{hyp2}, we quickly obtain ($n \geq 0$)
\begin{equation}\label{hyp3}
\sum_{m=0}^n (-1)^{m} \frac{c_m(x)}{m!} H_{n+1-m}^{(-x)} = \delta_{n,0},
\end{equation}
and
\begin{equation}\label{hyp4}
\sum_{m=0}^n (-1)^{m} \frac{\widehat{c}_{m}(x)}{m!} H_{n+1-m}^{(x)} = 1.
\end{equation}
Furthermore, using \eqref{rep2} in \eqref{hyp3} and \eqref{hyp4}, we arrive at the following recurrence formulas for Cauchy polynomials ($n \geq 1$):
\begin{equation*}\label{conec1}
\frac{c_n(x)}{n!} = \sum_{m=0}^{n-1} \frac{c_m(x)}{m!} \sum_{t=0}^{n-m}
\frac{(-1)^{n+1-m-t}}{n+1-m-t} \binom{x}{t},
\end{equation*}
and
\begin{equation*}\label{conec2}
\frac{\widehat{c}_n(x)}{n!} = (-1)^n + \sum_{m=0}^{n-1} \frac{\widehat{c}_m(x)}{m!} \sum_{t=0}^{n-m}
\frac{(-1)^{n+1-m}}{n+1-m-t} \binom{x+t-1}{t}.
\end{equation*}
Note that when $x=1$, the~last equation reduces to the identity in \eqref{zhao}.

\item
Clearly, more appealing identities can be derived from \eqref{hyp1} and \eqref{hyp2}. For~example, by~taking $y =-1$ and noting that~\cite{cere} (p.\ 18)
\begin{equation*}
H_{n+1}^{(-1)} = \begin{cases}
  1, & \text{for $n=0$;} \\
  \displaystyle{- \frac{1}{n(n+1)}}, & \text{for $n \geq 1$,}
\end{cases}
\end{equation*}
from \eqref{hyp1} and \eqref{hyp2}, we find the recurrence formulas ($n \geq 1$)
\begin{equation*}
\frac{c_n(x)}{n!} = (-1)^n \binom{x+n-2}{n} + \sum_{m=0}^{n-1} \frac{c_m(x)}{m!}
\, \frac{(-1)^{n-m}}{(n-m)(n+1-m)},
\end{equation*}
and
\begin{equation}\label{hyp5}
\frac{\widehat{c}_n(x)}{n!} = \binom{x}{n} + \sum_{m=0}^{n-1} \frac{\widehat{c}_m(x)}{m!}
\, \frac{(-1)^{n-m}}{(n-m)(n+1-m)}.
\end{equation}
Additionally, by~taking $x=1$ in \eqref{hyp5}, we obtain
\begin{equation*}
\frac{c_n}{n!} = \delta_{n,1} + \sum_{m=0}^{n-1} \frac{c_m}{m!} \, \frac{(-1)^{n-m}}{(n-m)(n+1-m)}, \quad n \geq 1.
\end{equation*}
\end{itemize}

By means of Proposition \ref{prop:2}, we also obtain the following relationships between Cauchy polynomials of the first and second kind.
\begin{proposition}\label{prop:3}
For integers $n \geq 0$ and for arbitrary $x$, we have
\begin{equation}\label{hyp6}
\frac{c_n(x)}{n!} = \sum_{m=0}^n (-1)^m \frac{\widehat{c}_m(x+n)}{m!},
\end{equation}
and
\begin{equation}\label{hyp7}
\frac{\widehat{c}_n(x)}{n!} = \sum_{m=0}^n (-1)^m \frac{c_m(x-n)}{m!}.
\end{equation}
\end{proposition}
\begin{proof}
We prove only \eqref{hyp6}. For~this, integrate over $y$ from $0$ to $1$ on both sides of \eqref{hyp2} to~obtain
\begin{equation*}
\sum_{m=0}^n (-1)^{m} \frac{\widehat{c}_m(x)}{m!} \int_{0}^1 H_{n+1-m}^{(y)} dy
= \int_{0}^1 \binom{-x+y+n}{n} dy = \frac{c_n(x-n)}{n!}.
\end{equation*}
{Hence,} using the representation $H_{j+1}^{(y)} = \frac{d}{dy} \binom{y+j}{j+1}$ (cf.\ \cite{cere}, Equation~(5)), and~setting $x \to x+n$, we obtain \eqref{hyp6}.
\end{proof}

\begin{remark}
Recalling that $c_n(2) = (-1)^n n! \sum_{m=0}^n (-1)^m \frac{\widehat{c}_m}{m!}$ and $\widehat{c}_n = (-1)^n n! \sum_{m=0}^n (-1)^m \frac{c_m}{m!}$, we can use \eqref{hyp6} (with $x = -n$) and \eqref{hyp7} (with $x=n$) to obtain the identities
\begin{equation*}
c_n(-n) = (-1)^n c_n(2) \quad\text{and}\quad \widehat{c}_n(n) = (-1)^n \widehat{c}_n.
\end{equation*}
\end{remark}

We finish this section by giving the following formulas for the hyperharmonic polynomials $H_{n}^{(x)}$ and $H_{n}^{(x+1)}$ in terms of Bernoulli polynomials and the corresponding Cauchy numbers ($n \geq 1$):
\begin{equation*}
H_{n}^{(x)} = \frac{1}{n!} \sum_{m=1}^n (-1)^{n-m} \binom{n}{m} c_{n-m}
\sum_{i=1}^m \genfrac{[}{]}{0pt}{}{m+1}{i+1} i B_{i-1}(x),
\end{equation*}
and
\begin{equation*}
H_{n}^{(x+1)} = \frac{1}{n!} \sum_{m=1}^n (-1)^{n-m} \binom{n}{m} \widehat{c}_{n-m}
\sum_{i=1}^m \genfrac{[}{]}{0pt}{}{m+1}{i+1} i B_{i-1}(x).
\end{equation*}

\section{Several Formulas Involving Bernoulli and Cauchy~Polynomials}\label{sec:5}

Recall from Theorem \ref{th:1} that Cauchy numbers of the first kind can be expressed in terms of Bernoulli numbers as
\begin{equation*}
c_n = \delta_{n,1} + (-1)^{n+1} n \sum_{m=1}^n \genfrac{[}{]}{0pt}{}{n-1}{m-1} \frac{B_m}{m},
\quad n \geq 1.
\end{equation*}
{Furthermore,} from~Theorem \ref{th:4} we quickly obtain
\begin{equation*}
\widehat{c}_{2n}(n) = -n \sum_{m=1}^n \frac{u(n,m)}{m} B_{2m}, \quad n \geq 1.
\end{equation*}

Our goal in this section is to derive several more formulas involving Bernoulli numbers and polynomials, Cauchy numbers and polynomials, and~GSNs of both kinds. Firstly, we derive a couple of explicit formulas for Bernoulli polynomials in terms of the corresponding Cauchy equivalents and the GSN of the second kind.
\begin{theorem}\label{th:7}
For integers $n \geq 1$ and for arbitrary $x$, we have
\begin{equation}
-\frac{B_{n}(x)}{n} = \sum_{m=1}^n \frac{c_{m}(x)}{m} \genfrac{\{}{\}}{0pt}{}{n-1}{m-1}_x, \label{th31}
\end{equation}
and
\begin{equation}
\frac{(-1)^{n}- B_{n}(x)}{n} = \sum_{m=1}^n \frac{\widehat{c}_{m}(-x)}{m}
\genfrac{\{}{\}}{0pt}{}{n-1}{m-1}_x, \label{th32}
\end{equation}
where $\genfrac{\{}{\}}{0pt}{}{n}{m}_x$ is the polynomial defined in \eqref{gs2}.
\end{theorem}
\begin{proof}
The proof of the theorem relies on the following formula for $S_n(x)$ \cite{cere4} (Lemma 1):
\begin{equation}\label{th3p1}
S_n(x) = S_n(a-1) + \sum_{m=0}^{n} m! \binom{x+1-a}{m+1} \genfrac{\{}{\}}{0pt}{}{n}{m}_a,  \quad n \geq 0,
\end{equation}
where $a$ stands for a fixed (but arbitrary) real number. To~derive \eqref{th31}, make the replacements $n \to n-1$ and $x \to x-1$ in \eqref{th3p1} and integrate over $x$ from $0$ to $1$ on both sides to obtain
\begin{equation*}
\int_{0}^1 S_{n-1}(x-1) dx - S_{n-1}(a-1) = \sum_{m=1}^{n} (m-1)! \genfrac{\{}{\}}{0pt}{}{n-1}{m-1}_a
\int_{0}^1 \binom{x-a}{m} dx.
\end{equation*}
{Applying} \eqref{lm11} with $y=0$ gives $\int_{0}^1 S_{n-1}(x-1) dx = -\frac{B_{n}(1)}{n}$ for $n \geq 1$. Moreover, from~\eqref{defb} we have $S_{n-1}(a-1) = \frac{1}{n}\big(B_{n}(a) - B_{n}(1)\big)$. Hence, recalling the definition of $c_n(x)$ in \eqref{defcaup}, we deduce that
\begin{equation*}
-\frac{B_{n}(a)}{n} = \sum_{m=1}^{n} \frac{c_{m}(a)}{m} \genfrac{\{}{\}}{0pt}{}{n-1}{m-1}_a,
\end{equation*}
which is Equation~\eqref{th31} with $x$ replaced by $a$.

Similarly, by~setting $n \to n-1$ and $x \to -x-1$ in \eqref{th3p1} and integrating both sides from $0$ to $1$, we obtain
\begin{equation*}
\int_{0}^1 S_{n-1}(-x-1) dx - S_{n-1}(a-1) = \sum_{m=1}^{n} (m-1)! \genfrac{\{}{\}}{0pt}{}{n-1}{m-1}_{a} \int_{0}^1 \binom{-a-x}{m} dx.
\end{equation*}
{By} \eqref{lm12} (with $y =0$), we have $\int_{0}^1 S_{n-1}(-x-1) dx = \frac{(-1)^n - B_{n}(1)}{n}$ for $n \geq 1$. Thus, recalling the definition of $\widehat{c}_n(x)$ in \eqref{defcaup}, it follows that
\begin{equation*}
\frac{(-1)^n - B_{n}(a)}{n} = \sum_{m=1}^{n} \frac{\widehat{c}_{m}(-a)}{m} \genfrac{\{}{\}}{0pt}{}{n-1}{m-1}_{a},
\end{equation*}
which is Equation~\eqref{th32} with $x$ replaced by $a$.
\end{proof}

Let us observe the following consequences of Theorem \ref{th:7}:
\begin{itemize}
\item
Setting $x=1$ and $x=0$ in \eqref{th31} gives (cf.\ \cite{merlini}, Theorem 2.9)
\begin{equation*}
-\frac{B_{n}}{n} = (-1)^n \sum_{m=1}^n \frac{\widehat{c}_{m}}{m} \genfrac{\{}{\}}{0pt}{}{n}{m}
\quad\text{and}\quad
-\frac{B_{n}}{n} = \sum_{m=1}^n \frac{c_{m}}{m} \genfrac{\{}{\}}{0pt}{}{n-1}{m-1}.
\end{equation*}

\item
Inverting \eqref{th31} and \eqref{th32} produces
\begin{equation}\label{inv2}
-\frac{c_{n}(x)}{n} = \sum_{m=1}^n \frac{(-1)^{n-m}}{m} \genfrac{[}{]}{0pt}{}{n-1}{m-1}_x B_{m}(x),
\end{equation}
and
\begin{equation}\label{inv3}
-\frac{\widehat{c}_{n}(-x)}{n} = \sum_{m=1}^n \frac{(-1)^{n}}{m} \genfrac{[}{]}{0pt}{}{n-1}{m-1}_x
\big( (-1)^m B_{m}(x)- 1\big).
\end{equation}
Upon using \eqref{th12}, this last equation can be written as
\begin{equation*}
-\frac{\widehat{c}_{n}(-x)}{n} = \widehat{c}_{n-1}(-x) + \sum_{m=1}^n \frac{(-1)^{n-m}}{m}
\genfrac{[}{]}{0pt}{}{n-1}{m-1}_x  B_{m}(x).
\end{equation*}

\item
For $x =1$ and $x=0$, \eqref{inv2} gives (cf.\ \cite{merlini}, Theorem 2.10)
\begin{equation}\label{exp5}
\frac{(-1)^{n+1} \widehat{c}_n}{n} = \sum_{m=1}^n \genfrac{[}{]}{0pt}{}{n}{m} \frac{B_m}{m}
\quad\text{and}\quad
\frac{(-1)^{n+1} c_n}{n} = \sum_{m=1}^n (-1)^m \genfrac{[}{]}{0pt}{}{n-1}{m-1} \frac{B_m}{m}.
\end{equation}

\item
The simultaneous validity of \eqref{exp3} and the second identity in \eqref{exp5} for $n \geq 2$ implies the~equality
\begin{equation*}
\sum_{m=1}^n \genfrac{[}{]}{0pt}{}{n-1}{m-1} \frac{B_m}{m} = \sum_{m=1}^n (-1)^m \genfrac{[}{]}{0pt}{}{n-1}{m-1}
\frac{B_m}{m}, \quad n \geq 2,
\end{equation*}
from which one can deduce inductively that $B_{2m+1} =0$ for all $m \geq 1$.
\end{itemize}

Interestingly, as~the next theorem shows, Bernoulli polynomials can be expressed in terms of the corresponding Cauchy numbers and the GSN of the second kind.
\begin{theorem}\label{th:8}
For integers $n \geq 1$ and for arbitrary $x$, we have
\begin{equation}
B_{n}(x) = x^n - n \sum_{m=1}^n \frac{c_{m}}{m} \genfrac{\{}{\}}{0pt}{}{n-1}{m-1}_x, \label{th41}
\end{equation}
and
\begin{equation}
B_{n}(x) = (x-1)^n - n \sum_{m=1}^n \frac{\widehat{c}_{m}}{m} \genfrac{\{}{\}}{0pt}{}{n-1}{m-1}_x , \label{th42}
\end{equation}
where $\genfrac{\{}{\}}{0pt}{}{n}{m}_x$ is the polynomial defined in \eqref{gs2}.
\end{theorem}
\begin{proof}
Set $n \to n-1$ and $x \to x+a-1$ in \eqref{th3p1} to obtain
\begin{equation*}
S_{n-1}(x+a-1) - S_{n-1}(a-1) = \sum_{m=1}^{n} (m-1)! \binom{x}{m} \genfrac{\{}{\}}{0pt}{}{n-1}{m-1}_a.
\end{equation*}
{Thus,} integrating both sides over the unit interval, we have
\begin{equation*}
\int_{0}^{1} S_{n-1}(x+a-1) dx - S_{n-1}(a-1) = \sum_{m=1}^{n} \frac{c_m}{m} \genfrac{\{}{\}}{0pt}{}{n-1}{m-1}_a.
\end{equation*}
{By} \eqref{lm11} (with $y =a$), $\int_{0}^1 S_{n-1}(x+a-1) dx$ is equal to $\frac{a^n - B_{n}(1)}{n}$ for $n \geq 1$. Then, using that $S_{n-1}(a-1) = \frac{1}{n}\big( B_{n}(a) - B_{n}(1) \big)$, it follows that
\begin{equation*}
B_n(a) = a^n - n \sum_{m=1}^{n} \frac{c_m}{m} \genfrac{\{}{\}}{0pt}{}{n-1}{m-1}_a,
\end{equation*}
which is Equation~\eqref{th41} after renaming $a$ to $x$.

On the other hand, setting $n \to n-1$ and $x \to -x+a-1$ in \eqref{th3p1} gives
\begin{equation*}
S_{n-1}(-x+a-1) - S_{n-1}(a-1) = \sum_{m=1}^{n} (m-1)! \binom{-x}{m} \genfrac{\{}{\}}{0pt}{}{n-1}{m-1}_a.
\end{equation*}
{Therefore,} integrating from $0$ to $1$ on both sides and noting that $\int_{0}^1 S_{n-1}(-x+a-1) dx = \frac{(a-1)^n - B_{n}(1)}{n}$ for $n \geq 1$, we arrive at Equation~\eqref{th42} after renaming $a$ to $x$.
\end{proof}

Setting $x=1$ in \eqref{th41} and $x=0$ in \eqref{th42}, we obtain the following couple of alternative formulas for Bernoulli numbers:
\begin{equation*}
B_n = (-1)^n \left( 1- n \sum_{m=1}^n \frac{c_{m}}{m} \genfrac{\{}{\}}{0pt}{}{n}{m} \right),
\end{equation*}
and
\begin{equation*}
B_n = (-1)^n - n \sum_{m=1}^n \frac{\widehat{c}_{m}}{m} \genfrac{\{}{\}}{0pt}{}{n-1}{m-1}.
\end{equation*}

Furthermore, combining \eqref{th31} and \eqref{th41} yields
\begin{equation*}
\sum_{m=1}^n \frac{c_m - c_m(x)}{m} \genfrac{\{}{\}}{0pt}{}{n-1}{m-1}_x = \frac{x^n}{n}, \quad n \geq 1,
\end{equation*}
and, in~particular,
\begin{equation*}
\sum_{m=1}^n \frac{c_m - \widehat{c}_m}{m} \genfrac{\{}{\}}{0pt}{}{n}{m} = \frac{1}{n}, \quad n \geq 1.
\end{equation*}
{Recalling} that $c_m - \widehat{c}_m = m \widehat{c}_{m-1}$, the~last equation can be written as
\begin{equation*}
\sum_{m=0}^n \genfrac{\{}{\}}{0pt}{}{n+1}{m+1} \widehat{c}_m = \frac{1}{n+1}, \quad n \geq 0,
\end{equation*}
which may also be obtained by inverting \eqref{kargin}.

Next, we derive the following relationships between Bernoulli and Cauchy polynomials.
\begin{proposition}\label{prop:4}
For integers $n \geq 0$ and for arbitrary $x$ and $y$, we have
\begin{align*}
B_n(x) & = \sum_{m=0}^n \sum_{l=0}^m  (-1)^m m! \genfrac{\{}{\}}{0pt}{}{n}{m}_x
\genfrac{\{}{\}}{0pt}{}{m}{l}_y  c_{l}(y),  \\
B_n(x) & = \sum_{m=0}^n \sum_{l=0}^m  m! \genfrac{\{}{\}}{0pt}{}{n}{m}_x
\genfrac{\{}{\}}{0pt}{}{m}{l}_y  \widehat{c}_{l}(-y),  \\
c_n(x) & =  \sum_{m=0}^n \sum_{l=0}^m  \frac{(-1)^{n-m+l}}{m!} \genfrac{[}{]}{0pt}{}{n}{m}_x
\genfrac{[}{]}{0pt}{}{m}{l}_y  B_{l}(y),  \\
\widehat{c}_n(-x) & =  \sum_{m=0}^n \sum_{l=0}^m  \frac{(-1)^{n-l}}{m!} \genfrac{[}{]}{0pt}{}{n}{m}_x
\genfrac{[}{]}{0pt}{}{m}{l}_y  B_{l}(y).
\end{align*}
\end{proposition}
\begin{proof}
We prove the first and third identities. For~this, we employ the following well-known explicit formula for Bernoulli polynomials~\cite{euler1} (Equation~(18):
\begin{equation*}
B_n(x) = \sum_{m=0}^n \sum_{l=0}^m \frac{(-1)^l}{m+1} \binom{m}{l} (x+l)^n.
\end{equation*}
{By} \eqref{gs22}, it is equal to
\begin{equation}\label{th51}
B_n(x) = \sum_{m=0}^n \frac{(-1)^m \, m!}{m+1} \genfrac{\{}{\}}{0pt}{}{n}{m}_x.
\end{equation}
{Moreover,} from~the first identity in \eqref{rem4}, we have
\begin{equation*}
\frac{1}{m+1} = \sum_{l=0}^m \genfrac{\{}{\}}{0pt}{}{m}{l}_y  c_l(y),
\end{equation*}
which holds for arbitrary $y$. Hence, substituting this expression for $\frac{1}{m+1}$ into \eqref{th51}, we easily derive the first identity of Proposition \ref{prop:4}. On~the other hand, by~inverting \eqref{th51}, we obtain
\begin{equation*}
\frac{1}{m+1} = \frac{1}{m!} \sum_{l=0}^m (-1)^l  \genfrac{[}{]}{0pt}{}{m}{l}_y  B_l(y),
\end{equation*}
which holds for arbitrary $y$. Using this expression for $\frac{1}{m+1}$ in \eqref{th11}, we end up with the third identity of Proposition \ref{prop:4}.
\end{proof}

Similarly, relying on the previous Theorems \ref{th:7} and \ref{th:8}, and using the identities in \eqref{rem4}, the following formulas expressing the Bernoulli polynomials in terms of Cauchy polynomials and the GSN of the second kind can be obtained.
\begin{proposition}\label{prop:5}
For integers $n \geq 1$ and for arbitrary $x$ and $y$, we have
\begin{align*}
-\frac{B_n(x)}{n} & = \sum_{m=1}^n \sum_{l=1}^m \genfrac{\{}{\}}{0pt}{}{n-1}{m-1}_x
\genfrac{\{}{\}}{0pt}{}{m-1}{l-1}_y c_m(x) c_{l-1}(y),  \\
\frac{B_n(x)}{n} & = \sum_{m=1}^n \sum_{l=1}^m (-1)^m \genfrac{\{}{\}}{0pt}{}{n-1}{m-1}_x
\genfrac{\{}{\}}{0pt}{}{m-1}{l-1}_y c_m(x)  \widehat{c}_{l-1}(-y), \\
\frac{(-1)^n -B_n(x)}{n} & = \sum_{m=1}^n \sum_{l=1}^m \genfrac{\{}{\}}{0pt}{}{n-1}{m-1}_x
\genfrac{\{}{\}}{0pt}{}{m-1}{l-1}_y \widehat{c}_m(-x) c_{l-1}(y), \\
\frac{B_n(x) - (-1)^n}{n} & = \sum_{m=1}^n \sum_{l=1}^m (-1)^m \genfrac{\{}{\}}{0pt}{}{n-1}{m-1}_x
\genfrac{\{}{\}}{0pt}{}{m-1}{l-1}_y \widehat{c}_m(-x) \widehat{c}_{l-1}(-y),
\end{align*}
and
\begin{align*}
B_n(x) & = x^n - n\sum_{m=1}^n \sum_{l=1}^m \genfrac{\{}{\}}{0pt}{}{n-1}{m-1}_x
\genfrac{\{}{\}}{0pt}{}{m-1}{l-1}_y c_m \, c_{l-1}(y),  \\
B_n(x) & = x^n + n\sum_{m=1}^n \sum_{l=1}^m  (-1)^m \genfrac{\{}{\}}{0pt}{}{n-1}{m-1}_x
\genfrac{\{}{\}}{0pt}{}{m-1}{l-1}_y  c_m \, \widehat{c}_{l-1}(-y),  \\
B_n(x) & = (x-1)^n - n \sum_{m=1}^n \sum_{l=1}^m  \genfrac{\{}{\}}{0pt}{}{n-1}{m-1}_x
\genfrac{\{}{\}}{0pt}{}{m-1}{l-1}_y \widehat{c}_m  \, c_{l-1}(y),  \\
B_n(x) & = (x-1)^n + n \sum_{m=1}^n \sum_{l=1}^m (-1)^m \genfrac{\{}{\}}{0pt}{}{n-1}{m-1}_x
\genfrac{\{}{\}}{0pt}{}{m-1}{l-1}_y \widehat{c}_m  \, \widehat{c}_{l-1}(-y).
\end{align*}
\end{proposition}

\begin{remark}
By assigning the values $(0,0), (0,1), (1,0)$, or~$(1,1)$ to the corresponding variables $(x,y)$ in Propositions \ref{prop:4} and \ref{prop:5}, one can find numerous relationships between Bernoulli and Cauchy~numbers.
\end{remark}

\section{Formulas Involving Poly-Cauchy Polynomials, Generalized Stirling Numbers, Binomial Coefficients and~More}\label{sec:6}

For integers $n \geq 0$ and $k \geq 1$, Kamano and Komatsu~\cite{kamano} introduced poly-Cauchy polynomials of the first and second kind, $c_n^{(k)}(x)$ and $\widehat{c}_n^{(k)}(x)$, as~a generalization of classical Cauchy polynomials. These can be defined by integral formulas (see~\cite{kamano} (Sections 2 and 3) and~\cite{komatsu4} (Theorems 3 and 9))
\begin{equation}\label{defpol1}
c_n^{(k)}(x) = n! \int_{0}^{1} \ldots \int_{0}^{1} \binom{t_1 \cdots t_k - x}{n}
dt_1 \cdots dt_k,
\end{equation}
and
\begin{equation}\label{defpol2}
\widehat{c}_n^{(k)}(x) = n! \int_{0}^{1} \ldots \int_{0}^{1} \binom{x - t_1 \cdots t_k}{n}
dt_1 \cdots dt_k.
\end{equation}
{(Note that} $x$ is replaced by $-x$ in the original definition of poly-Cauchy polynomials set forth in~\cite{kamano}). Clearly, poly-Cauchy polynomials reduce to ordinary Cauchy polynomials when $k=1$. For~$x=0$, $c_n^{(k)} = c_n^{(k)}(0)$ and $\widehat{c}_n^{(k)} = \widehat{c}_n^{(k)}(0)$ are called, respectively, poly-Cauchy numbers of the first and second kind~\cite{komatsu}. Notice that for the general case where $k >1$, $c_n^{(k)}(x)$ is no longer equal to $\widehat{c}_n^{(k)}(1-x)$.

As an explicit example, for~$n=6$, the~poly-Cauchy polynomials of the first and second kind are given by
\begin{align*}
c_6^{(k)}(x)  = & -\frac{120}{2^k} +\frac{274}{3^k} -\frac{225}{4^k} +\frac{85}{5^k} -\frac{15}{6^k} +\frac{1}{7^k} \\
& + \left( 120 -\frac{548}{2^k} +\frac{675}{3^k} -\frac{340}{4^k} +\frac{75}{5^k} -\frac{6}{6^k} \right) x \\
& + \left( 274 -\frac{675}{2^k} +\frac{510}{3^k} -\frac{150}{4^k} +\frac{15}{5^k} \right) x^2
+ \left( 225 -\frac{340}{2^k} +\frac{150}{3^k} -\frac{20}{4^k} \right) x^3 \\
& + \left( 85 -\frac{75}{2^k} +\frac{15}{3^k} \right) x^4 + \left( 15 -\frac{6}{2^k}\right) x^5 + x^6, \\[-2mm]
\intertext{and}
\widehat{c}_6^{(k)}(x) = & +\frac{120}{2^k} +\frac{274}{3^k} +\frac{225}{4^k} +\frac{85}{5^k} +\frac{15}{6^k} +\frac{1}{7^k} \\
& - \left( 120 +\frac{548}{2^k} +\frac{675}{3^k} +\frac{340}{4^k} +\frac{75}{5^k} +\frac{6}{6^k} \right) x \\
& + \left( 274 +\frac{675}{2^k} +\frac{510}{3^k} +\frac{150}{4^k} +\frac{15}{5^k} \right) x^2
- \left( 225 +\frac{340}{2^k} +\frac{150}{3^k} +\frac{20}{4^k} \right) x^3 \\
& + \left( 85 +\frac{75}{2^k} +\frac{15}{3^k} \right) x^4 - \left( 15 +\frac{6}{2^k}\right) x^5 + x^6.
\end{align*}

In this section, we provide several formulas and identities for poly-Cauchy polynomials by generalizing the previous results from Sections~\ref{sec:2}, \ref{sec:3} and~\ref{sec:5}.

\subsection{Some Basic Formulas for Poly-Cauchy~Polynomials}

The following formulas generalize those in Proposition \ref{prop:1} to poly-Cauchy polynomials.
\begin{proposition}\label{prop:6}
For integers $n \geq 0$ and for arbitrary $x$, we have
\begin{equation}\label{pro41}
c_n^{(k)}(x) = \sum_{m=0}^n \frac{(-1)^{n-m}}{(m+1)^k} \genfrac{[}{]}{0pt}{}{n}{m}_x,
\end{equation}
and
\begin{equation}\label{pro42}
\widehat{c}_n^{(k)}(-x) = (-1)^n \sum_{m=0}^n \frac{1}{(m+1)^k} \genfrac{[}{]}{0pt}{}{n}{m}_x,
\end{equation}
where $\genfrac{[}{]}{0pt}{}{n}{m}_x$ is the polynomial defined in \eqref{gs1}.
\end{proposition}
\begin{proof}
The above formulas can be proved in the same way as the formulas in Proposition~\ref{prop:1} once we make the identifications $t \equiv t_1 \cdots t_k$ and $dt \equiv dt_1 \ldots  dt_k$. So, writing \eqref{th1p1} in the~form
\begin{equation*}
\prod_{l=1}^{n} (t_1 \cdots t_k -x+l) = \sum_{m=0}^{n} P_{n+1,m}(x) \, t_1^m \cdots t_k^m,
\end{equation*}
and noting that
\begin{equation*}
\prod_{l=1}^{n} (t_1 \cdots t_k -x+l) = n! \binom{t_1 \cdots t_k -x+n}{n} = (-1)^n n!
\binom{x - t_1 \cdots t_k -1}{n},
\end{equation*}
it follows that
\begin{equation*}
\widehat{c}_n^{(k)}(x) = (-1)^n \sum_{m=0}^n P_{n+1,m}(x+1) \! \int_{0}^{1} t_1^m dt_1 \, \ldots
\int_{0}^{1} t_k^m dt_k = (-1)^n \sum_{m=0}^n \frac{P_{n+1,m}(x+1)}{(m+1)^k},
\end{equation*}
which amounts to \eqref{pro42} due to the relation $P_{n+1,m}(x+1) = \genfrac{[}{]}{0pt}{}{n}{m}_{-x}$.
\end{proof}

An immediate consequence of Proposition \ref{prop:6} is the following.
\begin{corollary}\label{col:2}
Let $c_n^{(k)}(x) = \sum_{i=0}^n c_{n,i}^{(k)} x^i$ and $\widehat{c}_n^{(k)}(x) = \sum_{i=0}^n \widehat{c}_{n,i}^{(k)} x^i$ be poly-Cauchy polynomials of the first and second kind, respectively. Then, for~$i = 0,1,\ldots,n$, we have
\begin{equation*}
c_{n,i}^{(k)} = (-1)^{n+i} \sum_{m=i}^n \frac{(-1)^m}{(m-i+1)^k} \binom{m}{i} \genfrac{[}{]}{0pt}{}{n}{m},
\end{equation*}
and
\begin{equation*}
\widehat{c}_{n,i}^{(k)} = (-1)^{n+i} \sum_{m=i}^n \frac{1}{(m-i+1)^k} \binom{m}{i} \genfrac{[}{]}{0pt}{}{n}{m}.
\end{equation*}
\end{corollary}

\begin{remark}
The polynomials $c_n^{(k)}(x)$ and $\widehat{c}_n^{(k)}(x)$ defined in Corollary \ref{col:2} can be expressed equivalently as
\begin{equation}\label{back1}
c_n^{(k)}(x) = \sum_{m=0}^n (-1)^{n-m} \genfrac{[}{]}{0pt}{}{n}{m} \sum_{i=0}^m \binom{m}{i} \frac{(-x)^i}{(m-i+1)^k},
\end{equation}
and
\begin{equation}\label{back2}
\widehat{c}_n^{(k)}(x) = (-1)^n \sum_{m=0}^n \genfrac{[}{]}{0pt}{}{n}{m} \sum_{i=0}^m \binom{m}{i} \frac{(-x)^i}{(m-i+1)^k},
\end{equation}
in accordance with~\cite{kamano} (Theorems 1 and 4) and~\cite{komatsu4} (Theorems 1 and 7).
\end{remark}

Likewise, several of the other formulas we have previously obtained for Cauchy polynomials can be extended to poly-Cauchy polynomials by simply changing $c_n(x)$ to $ c_n^{(k)}(x)$, $\widehat{c}_n(x)$ to $\widehat{c}_n^{(k)}(x)$, and~by raising to the $k$-th power the appropriate denominators (if any). Next, we write down the generalization for poly-Cauchy polynomials of Equations~\eqref{diff1}, \eqref{diff2}, \eqref{whit1} and \eqref{whit2}, Theorem \ref{th:2}, Equations~\eqref{symm1} and \eqref{symm2}, and~Theorems \ref{th:3}, \ref{th:5} and~\ref{th:6}. The~resulting formulas are, respectively, as~follows:
\begingroup
\allowdisplaybreaks
\begin{align}
c^{(k)}_{n}( & x +1) - c^{(k)}_{n}(x) = -n c^{(k)}_{n-1}(x+1), \notag \\[1mm]
\widehat{c}^{(k)}_{n}( & x +1) - \widehat{c}^{(k)}_{n}(x) = n \widehat{c}^{(k)}_{n-1}(x), \notag \\
c_n^{(k)}\left( \frac{r}{m} \right) & = \sum_{l=0}^n \frac{(-1)^{n-l}}{(l+1)^k} \,
\frac{w_{m,r}(n,l)}{m^{n-l}}, \notag \\
\widehat{c}_n^{(k)}\left( -\frac{r}{m} \right) & = (-1)^n  \sum_{l=0}^n \frac{1}{(l+1)^k} \,
\frac{w_{m,r}(n,l)}{m^{n-l}}, \notag \\
c_n^{(k)}(x) & = (-1)^n n! \sum_{m=0}^n \frac{\widehat{c}_m^{(k)}}{m!} \binom{x+n-1}{n-m}, \label{symm5} \\
\widehat{c}_n^{(k)}(x) & = n! \sum_{m=0}^n \frac{(-1)^m c_m^{(k)}}{m!} \binom{x-m}{n-m}, \label{symm6} \\
c_n^{(k)}(x) & = n! \sum_{m=0}^n \frac{c^{(k)}_m}{m!} \binom{-x}{n-m}
= \sum_{m=0}^n (-1)^m c_{n-m}^{(k)} \binom{n}{m} x^{(m)}, \label{symm7} \\
\widehat{c}^{(k)}_n(x) & = n! \sum_{m=0}^n \frac{\widehat{c}^{(k)}_m}{m!} \binom{x}{n-m}
= \sum_{m=0}^n \widehat{c}^{(k)}_{n-m} \binom{n}{m} (x)_m, \label{symm8} \\
c_{n+1}^{(k)}(x) & = -(n+x)c_n^{(k)}(x) + (-1)^{n+1} n! \sum_{m=0}^n
\frac{\widehat{c}_{m+1}^{(k)}}{m!} \binom{x+n}{n-m}, \notag \\
\widehat{c}_{n+1}^{(k)}(x) & = -(x-n)\widehat{c}_n^{(k)}(x) - n! \sum_{m=0}^n
\frac{(-1)^m \, c^{(k)}_{m+1}}{m!} \binom{x-m-1}{n-m}, \notag \\
\frac{d^i c^{(k)}_n(x)}{d x^i} & = (-1)^i i! \sum_{m=i}^n \binom{n}{m} \binom{m}{i}
\frac{B_{m-i}^{(n+1)}(1-x)}{(n+1-m)^k}, \label{genk1} \\
\frac{d^i \widehat{c}^{(k)}_n(x)}{d x^i} & = i! \sum_{m=i}^n (-1)^{n-m} \binom{n}{m} \binom{m}{i}
\frac{B_{m-i}^{(n+1)}(x+1)}{(n+1-m)^k}, \label{genk2} \\[-2mm]
\intertext{and}
\frac{d^i c_n^{(k)}(x)}{d x^i} & = i! \sum_{m=i}^n (-1)^m c^{(k)}_{n-m} \binom{n}{m}
\genfrac{[}{]}{0pt}{}{m}{i}_x, \notag \\
\frac{d^i \widehat{c}^{(k)}_n(x)}{d x^i} & = (-1)^i i! \sum_{m=i}^n (-1)^m \, \widehat{c}_{n-m}^{(k)}
\binom{n}{m} \genfrac{[}{]}{0pt}{}{m}{i}_{-x}, \notag
\end{align}
\endgroup
where, in~Equation~\eqref{symm7}, $x^{(m)}$ denotes the rising factorial $x^{(m)} = x(x+1)\ldots (x+m-1)$ for $m \geq 1$ and $x^{(0)} =1$.

Additionally, as~can be shown, the~following identities
\begin{equation} \label{korec1}
\sum_{m=r}^n \genfrac{\{}{\}}{0pt}{}{n-r}{m-r}_r c_{m-s}^{(k)}(s) = \sum_{l=s}^r \frac{(-1)^{r-l}}
{(n+l-r-s+1)^k} \genfrac{[}{]}{0pt}{}{r-s}{l-s}_s,
\end{equation}
and
\begin{equation} \label{korec2}
\sum_{m=r}^n \genfrac{\{}{\}}{0pt}{}{n-r}{m-r}_r \widehat{c}_{m-s}^{(k)}(-s) = (-1)^{n-s} \sum_{l=s}^r
\frac{1}{(n+l-r-s+1)^k} \genfrac{[}{]}{0pt}{}{r-s}{l-s}_s,
\end{equation}
hold for any integers $s$, $r$, and~$n$ fulfilling $0 \leq s \leq r \leq n$.

The following are a few remarks regarding the foregoing~equations.
\begin{itemize}
\item
By inverting \eqref{pro41} and \eqref{pro42} we obtain (cf.\ \cite{cenkci}, Theorem 3.3)
\begin{equation*}
\sum_{m=0}^n \genfrac{\{}{\}}{0pt}{}{n}{m}_x c_m^{(k)}(x) = \frac{1}{(n+1)^k}
\quad \, \text{and}\quad \sum_{m=0}^n \genfrac{\{}{\}}{0pt}{}{n}{m}_x \widehat{c}_m^{(k)}(-x)
= \frac{(-1)^n}{(n+1)^k}.
\end{equation*}

\item
From Equations~\eqref{symm5}--\eqref{symm8}, we find the identities
\begin{align*}
c_n^{(k)}(1) & = (-1)^n n! \sum_{m=0}^n \binom{n}{m} \frac{\widehat{c}_m^{(k)}}{m!}
= (-1)^n n! \sum_{m=0}^n \frac{(-1)^m \, c_m^{(k)}}{m!}, \\
\widehat{c}_n^{(k)}(-1) & = (-1)^n n! \sum_{m=0}^n \binom{n}{m} \frac{c_m^{(k)}}{m!}
= (-1)^n n! \sum_{m=0}^n \frac{(-1)^m \, \widehat{c}_m^{(k)}}{m!},
\end{align*}
along with
\begin{equation*}
c_n^{(k)} = (-1)^n \sum_{m=0}^n L(n,m) \, \widehat{c}_m^{(k)}
\quad\text{and}\quad \widehat{c}_n^{(k)} = (-1)^n \sum_{m=0}^n L(n,m) \, c_m^{(k)},
\end{equation*}
where $L(n,m) = \frac{n!}{m!} \binom{n-1}{m-1}$ are the (unsigned) Lah numbers~\cite{lah}.

\item
The expressions for $c_n^{(k)}(x)$ and $\widehat{c}_n^{(k)}(x)$ on the rightmost sides of \eqref{symm7} and \eqref{symm8} were previously obtained via umbral calculus in~\cite{kim3} (Theorem 2.10) and~\cite{kim4} (Theorem 7), respectively. On~the other hand, taking $i=0$ in \eqref{genk1} and \eqref{genk2} yields
\begin{equation*}
c^{(k)}_n(x) = \sum_{m=0}^n \binom{n}{m} \frac{B_{m}^{(n+1)}(1-x)}{(n+1-m)^k},
\end{equation*}
(\cite{komatsu4}, Theorem 12), and~\begin{equation*}
\widehat{c}^{(k)}_n(x) = \sum_{m=0}^n (-1)^{n-m} \binom{n}{m} \frac{B_{m}^{(n+1)}(x+1)}{(n+1-m)^k},
\end{equation*}
generalizing Formulas \eqref{gbp1} and \eqref{gbp2} to poly-Cauchy case.

\item
Letting $s=0$ in \eqref{korec1} and \eqref{korec2} leads to
\begin{equation*}
\sum_{m=0}^n \genfrac{\{}{\}}{0pt}{}{n}{m}_r c_{m+r}^{(k)} = \sum_{l=0}^r
\frac{(-1)^{r-l}}{(n+l+1)^k} \genfrac{[}{]}{0pt}{}{r}{l},
\end{equation*}
and
\begin{equation*}
\sum_{m=0}^n \genfrac{\{}{\}}{0pt}{}{n}{m}_r \widehat{c}_{m+r}^{(k)} = (-1)^{n+r}
\sum_{l=0}^r \frac{1}{(n+l+1)^k} \genfrac{[}{]}{0pt}{}{r}{l},
\end{equation*}
which are equivalent, respectively, to~the identities in Theorems 1 and 3 of~\cite{komatsu8}.

\end{itemize}

The following proposition extends the integration formula in \eqref{int1} to poly-Cauchy~polynomials.
\begin{proposition}\label{prop:7}
For integers $n \geq 0$ and $k \geq 1$, we have
\begin{equation}\label{int2}
\int_0^1 c_n^{(k)}(x) dx = c_n -n \sum_{j=1}^k c_n^{(j)},
\end{equation}
and
\begin{equation}\label{int3}
\int_0^1 \widehat{c}_n^{(k)}(x) dx = \widehat{c}_n -n \sum_{j=1}^k \big( \widehat{c}_n^{(j)}
+(n-1)\widehat{c}_{n-1}^{(j)} \big).
\end{equation}
\end{proposition}
\begin{proof}
For Formula \eqref{int2}, integrate \eqref{symm7} to obtain
\begin{align*}
\int_0^1 c_n^{(k)}(x) dx & = n! \sum_{m=0}^n \frac{c^{(k)}_m}{m!} \int_0^1 \binom{-x}{n-m} dx \\
& = \sum_{m=0}^n \binom{n}{m} c^{(k)}_m \, \widehat{c}_{n-m} = c_n -n \sum_{j=1}^k c_n^{(j)},
\end{align*}
where, in~the last step, we have used the identity in~\cite{komatsu7} (Equation~(99)). For~Formula \eqref{int3}, integrate \eqref{symm8} to obtain
\begin{align*}
\int_0^1 \widehat{c}_n^{(k)}(x) dx & = n! \sum_{m=0}^n \frac{\widehat{c}^{(k)}_m}{m!} \int_0^1 \binom{x}{n-m} dx \\
& = \sum_{m=0}^n \binom{n}{m} \widehat{c}^{(k)}_m \, c_{n-m} = \widehat{c}_n -n \sum_{j=1}^k
\big( \widehat{c}_n^{(j)} +(n-1)\widehat{c}_{n-1}^{(j)} \big),
\end{align*}
where, in~the last step, we have used the identity in~\cite{komatsu7} (Equation~(89)).
\end{proof}

On the other hand, the~generalization of Proposition \ref{prop:4} for poly-Cauchy polynomials can be accomplished by means of the corresponding poly-Bernoulli polynomials $\mathbb{B}_n^{(k)}(x)$~\mbox{\cite{coppo,bayad}}. These are generalizations of the so-called poly-Bernoulli numbers $\mathbb{B}_n^{(k)}$ introduced by Kaneko~\cite{kaneko}, where $\mathbb{B}_n^{(k)}= \mathbb{B}_n^{(k)}(0)$. Specifically, by~defining $\mathbb{B}_n^{(k)}(x)$ as
\begin{equation*}
\mathbb{B}_n^{(k)}(x) = (-1)^n \sum_{m=0}^n \frac{(-1)^m \, m!}{(m+1)^k} \genfrac{\{}{\}}{0pt}{}{n}{m}_x,
\end{equation*}
it can be shown that (cf.\ \cite{cenkci} (Theorem 3.4))
\begin{align*}
\mathbb{B}_n^{(k)}(x) & = (-1)^n \sum_{m=0}^n \sum_{l=0}^m  (-1)^{m} m! \genfrac{\{}{\}}{0pt}{}{n}{m}_x
\genfrac{\{}{\}}{0pt}{}{m}{l}_y  c_{l}^{(k)}(y), \\
\mathbb{B}_n^{(k)}(x) & = (-1)^n \sum_{m=0}^n \sum_{l=0}^m  m! \genfrac{\{}{\}}{0pt}{}{n}{m}_x
\genfrac{\{}{\}}{0pt}{}{m}{l}_y \widehat{c}_{l}^{(k)}(-y), \\
c_n^{(k)}(x) & = (-1)^n \sum_{m=0}^n \sum_{l=0}^m  \frac{(-1)^{m}}{m!} \genfrac{[}{]}{0pt}{}{n}{m}_x
\genfrac{[}{]}{0pt}{}{m}{l}_y  \mathbb{B}_{l}^{(k)}(y), \\
\widehat{c}_n^{(k)}(-x) & = (-1)^n \sum_{m=0}^n \sum_{l=0}^m  \frac{1}{m!} \genfrac{[}{]}{0pt}{}{n}{m}_x
\genfrac{[}{]}{0pt}{}{m}{l}_y \mathbb{B}_{l}^{(k)}(y),
\end{align*}
which hold for all integers $n \geq 0$, $k \geq 1$, and~for arbitrary $x$ and $y$.

On the other hand, Komatsu and Luca~\cite{luca} (p.\ 105) (see also~\cite{komatsu9}) defined different poly-Bernoulli polynomials by
\begin{equation*}
\mathbf{B}_n^{(k)}(x) = (-1)^n \sum_{m=0}^n (-1)^m m! \genfrac{\{}{\}}{0pt}{}{n}{m}
\sum_{i=0}^m \binom{m}{i} \frac{(-x)^i}{(m-i+1)^k}.
\end{equation*}
{As it} turns out, the~corresponding relationships between $\mathbf{B}_n^{(k)}(x)$ and poly-Cauchy polynomials are given by (cf.\ \cite{luca}, Theorem 4.1)
\begin{equation}\label{inform}
\begin{split}
\mathbf{B}_n^{(k)}(x) & = (-1)^n \sum_{m=0}^n \sum_{l=0}^m  (-1)^{m} m! \genfrac{\{}{\}}{0pt}{}{n}{m}
\genfrac{\{}{\}}{0pt}{}{m}{l} c_{l}^{(k)}(x), \\
\mathbf{B}_n^{(k)}(x) & = (-1)^n \sum_{m=0}^n \sum_{l=0}^m  m! \genfrac{\{}{\}}{0pt}{}{n}{m}
\genfrac{\{}{\}}{0pt}{}{m}{l} \widehat{c}_{l}^{(k)}(x), \\
c_n^{(k)}(x) & = (-1)^n \sum_{m=0}^n \sum_{l=0}^m  \frac{(-1)^{m}}{m!} \genfrac{[}{]}{0pt}{}{n}{m}
\genfrac{[}{]}{0pt}{}{m}{l}  \mathbf{B}_{l}^{(k)}(x), \\
\widehat{c}_n^{(k)}(x) & = (-1)^n \sum_{m=0}^n \sum_{l=0}^m  \frac{1}{m!} \genfrac{[}{]}{0pt}{}{n}{m}
\genfrac{[}{]}{0pt}{}{m}{l} \mathbf{B}_{l}^{(k)}(x),
\end{split}
\end{equation}
which hold for all integers $n \geq 0$, $k \geq 1$, and~for arbitrary $x$.

\subsection{Additional Formulas for Poly-Cauchy~Polynomials}

The following formulas for $c_{n}^{(k)}(x)$ and $\widehat{c}_{n}^{(k)}(x)$ are the generalizations for poly-Cauchy polynomials of the formulas for $c_{n}(x)$ and $\widehat{c}_{n}(x)$ given in Theorem \ref{th:1}.
\begin{theorem}\label{th:9}
For integers $n \geq 1$ and for arbitrary $x$, we have
\begin{equation}\label{poly1}
c_{n}^{(k)}(x) = \delta_{n,1} + (-1)^{n} n \sum_{m=1}^n \frac{1}{m} \genfrac{[}{]}{0pt}{}{n-1}{m-1}
\sum_{j=1}^m \binom{m}{j} B_{m-j} \, \mathbb{C}_j^{(k)}(x+1),
\end{equation}
and
\begin{equation}\label{poly2}
\widehat{c}_{n}^{(k)}(x) = \delta_{n,1} + (-1)^{n} n \sum_{m=1}^n \frac{1}{m} \genfrac{[}{]}{0pt}{}{n-1}{m-1}
\sum_{j=1}^m (-1)^j \binom{m}{j} B_{m-j} \, \mathbb{C}_j^{(k)}(x-1),
\end{equation}
where $\mathbb{C}_j^{(k)}(x)$ is the polynomial in $x$ of degree $j$ defined by
\begin{equation}\label{def1}
\mathbb{C}_j^{(k)}(x) = \begin{cases}
  1, & \text{for $j=0$;} \\
  \displaystyle{\sum_{i=0}^{j} \frac{(-1)^i}{(i+1)^k} \binom{j}{i} x^{j-i}},
& \text{for $j \geq 1$.}
\end{cases}
\end{equation}
\end{theorem}
\begin{proof}
We prove only \eqref{poly1}, but~the proof of \eqref{poly2} is similar. From~\eqref{poly3}, we see that, for~$n \geq 1$,
\begin{equation}\label{poly4}
\frac{c_{n}^{(k)}(x)}{n} = \delta_{n,1} + \sum_{m=1}^n (-1)^{n-m} \genfrac{[}{]}{0pt}{}{n-1}{m-1}
\int_{0}^{1} \ldots \int_{0}^{1} S_{m-1}(t_1 \cdots t_k -x-1) dt_1 \cdots dt_k.
\end{equation}
{Now,} invoking the well-known formula for the power sum polynomials (see, e.g.,~\cite{wu})
\begin{equation*}
S_{m-1}(x) = \frac{1}{m} \sum_{j=1}^m (-1)^{m-j} \binom{m}{j} B_{m-j} x^j, \quad m \geq 1,
\end{equation*}
we have
\begin{align*}
S_{m-1}(t_1 \cdots t_k -x-1) & = \sum_{j=1}^{m} \frac{(-1)^{m-j}}{m} \binom{m}{j}
B_{m-j} \times \big(t_1 \cdots t_k -x-1 \big)^j \\
& = \frac{(-1)^m}{m} \sum_{j=1}^{m} \binom{m}{j} B_{m-j} \sum_{i=0}^{j} (-1)^i
\binom{j}{i} (x+1)^{j-i} \big(t_1 \cdots t_k \big)^i,
\end{align*}
and then
\begin{equation*}
\int_{0}^{1} \ldots \int_{0}^{1} S_{m-1}(t_1 \cdots t_k -x-1) dt_1 \cdots dt_k
= \frac{(-1)^m}{m} \sum_{j=1}^{m} \binom{m}{j} B_{m-j} \, \mathbb{C}_j^{(k)}(x+1).
\end{equation*}
{Thus,} substituting the above integration formula into \eqref{poly4}, we obtain \eqref{poly1}.
\end{proof}

\begin{remark}
Thanks to \eqref{exp3}, Equations~\eqref{poly1} and \eqref{poly2} can be rewritten as
\begin{equation}\label{poly5}
c_{n}^{(k)}(x) = c_n + (-1)^{n} n \sum_{m=1}^n \frac{1}{m} \genfrac{[}{]}{0pt}{}{n-1}{m-1}
\sum_{j=0}^m \binom{m}{j} B_{m-j} \, \mathbb{C}_j^{(k)}(x+1),
\end{equation}
and
\begin{equation}\label{poly6}
\widehat{c}_{n}^{(k)}(x) = c_n + (-1)^{n} n \sum_{m=1}^n \frac{1}{m} \genfrac{[}{]}{0pt}{}{n-1}{m-1}
\sum_{j=0}^m (-1)^j \binom{m}{j} B_{m-j} \, \mathbb{C}_j^{(k)}(x-1).
\end{equation}
{Hence,} comparing \eqref{poly5} (for $k=1$) with \eqref{exp1}, and~\eqref{poly6} (for $k=1$) with \eqref{exp2}, we deduce that
\begin{equation}\label{idc1}
\sum_{j=0}^m \binom{m}{j} B_{m-j} \, \mathbb{C}_j^{(1)}(x+1) = x^m,
\end{equation}
and
\begin{equation}\label{idc2}
\sum_{j=0}^m (-1)^{m-j} \binom{m}{j} B_{m-j} \, \mathbb{C}_j^{(1)}(x) = x^m,
\end{equation}
which hold for arbitrary $x$ and for any non-negative integer $m$.
\end{remark}

The following theorem is the generalization of Equations~\eqref{inv2} and \eqref{inv3} to poly-Cauchy polynomials. We omit the proof for the sake of brevity.
\begin{theorem}\label{th:10}
For integers $n \geq 1$ and arbitrary $x$, we have
\begin{equation}\label{poly7}
-\frac{c_{n}^{(k)}(x)}{n} = \sum_{m=1}^n \frac{(-1)^{n}}{m} \genfrac{[}{]}{0pt}{}{n-1}{m-1}_x
\bigg( (-1)^m B_{m}(x) - \sum_{j=0}^m \binom{m}{j} B_{m-j} \, \mathbb{C}_{j}^{(k)}(1) \bigg),
\end{equation}
and
\begin{equation}\label{poly8}
-\frac{\widehat{c}_{n}^{(k)}(-x)}{n} = \sum_{m=1}^n \frac{(-1)^{n}}{m} \genfrac{[}{]}{0pt}{}{n-1}{m-1}_x
\bigg( (-1)^m B_{m}(x) - \sum_{j=0}^m (-1)^j \binom{m}{j} B_{m-j} \, \mathbb{C}_{j}^{(k)}(-1) \bigg),
\end{equation}
where $\mathbb{C}_j^{(k)}(x)$ is the polynomial defined in \eqref{def1}.
\end{theorem}

We note that employing Formulas \eqref{th41} and \eqref{th42} for Bernoulli polynomials in \eqref{poly7} and \eqref{poly8} leads, respectively, to~the alternative formulas ($n \geq 1$)
\begin{equation*}
c_{n}^{(k)}(x) = c_n -n \sum_{m=1}^n \frac{(-1)^{n}}{m} \genfrac{[}{]}{0pt}{}{n-1}{m-1}_x
\bigg((-x)^m - \sum_{j=0}^m \binom{m}{j} B_{m-j} \, \mathbb{C}_{j}^{(k)}(1) \bigg),
\end{equation*}
and
\begin{equation*}
\widehat{c}_{n}^{(k)}(-x) = \widehat{c}_n -n \sum_{m=1}^n \frac{(-1)^{n}}{m} \genfrac{[}{]}{0pt}{}{n-1}{m-1}_x
\bigg( (1-x)^m - \sum_{j=0}^m (-1)^j \binom{m}{j} B_{m-j} \, \mathbb{C}_{j}^{(k)}(-1) \bigg).
\end{equation*}
{For} $k=1$, the~above formulas become
\begin{equation*}
c_{n}(x) = c_n -(-1)^n n \sum_{m=1}^n \frac{(-1)^m}{m} \genfrac{[}{]}{0pt}{}{n-1}{m-1}_x  x^m,
\end{equation*}
and
\begin{equation*}
\widehat{c}_{n}(-x) = \widehat{c}_n - (-1)^n n \sum_{m=1}^n \frac{(-1)^m}{m} \genfrac{[}{]}{0pt}{}{n-1}{m-1}_{x}
\! \big( (x-1)^m - 1 \big),
\end{equation*}
which may be compared with the formulas in \eqref{exp1} and \eqref{exp2}.

Finally, the extension to poly-Cauchy polynomials of Theorem \ref{th:4} reads as follows.
\begin{theorem}\label{th:11}
For integers $n \geq 1$ and for arbitrary $x$, we have
\begin{align*}
c^{(k)}_{2n}(x) & = n \sum_{m=1}^n \frac{u(n,m)}{m} \sum_{j=0}^{2m-1} \binom{2m}{j} B_{j}
\, \mathbb{C}_{2m-j}^{(k)}(x+n), \\
\widehat{c}^{(k)}_{2n}(x) & = n \sum_{m=1}^n \frac{u(n,m)}{m} \sum_{j=0}^{2m-1} (-1)^j \binom{2m}{j} B_{j}
\, \mathbb{C}_{2m-j}^{(k)}(x-n), \\
c^{(k)}_{2n+1}(x) & = -(2n+1) \sum_{m=1}^n \frac{u(n,m)}{2m+1}
\sum_{j=0}^{2m} 2^{j} \binom{2m+1}{j} B_{j}\, \mathbb{C}_{2m+1-j}^{(k)}(x+n+1), \\
\widehat{c}^{(k)}_{2n+1}(x) & = (2n+1) \sum_{m=1}^n \frac{u(n,m)}{2m+1} \sum_{j=0}^{2m} 2^{j} \binom{2m+1}{j}
B_{j}\, \mathbb{C}_{2m+1-j}^{(k)}(x-n+1),
\end{align*}
where $u(n,m)$ are the central factorial numbers with even indices of the first kind and $\mathbb{C}_j^{(k)}(x)$ is the polynomial defined in \eqref{def1}.
\end{theorem}

We present several comments with respect to Theorem \ref{th:11}.
\begin{itemize}
\item
By \eqref{idc1} and \eqref{idc2}, the~formulas for $c^{(k)}_{2n}(x)$ and $\widehat{c}^{(k)}_{2n}(x)$ in Theorem \ref{th:11} reduce to the formulas for $c_{2n}(x)$ and $\widehat{c}_{2n}(x)$ in Theorem \ref{th:4} when $k=1$.

\item
Comparing the formula for $c^{(k)}_{2n+1}(x)$ in Theorem \ref{th:11} with that for $c_{2n+1}(x)$ in Theorem~\ref{th:4} reveals that
\begin{equation*}
E_{2m+1}(x) = \sum_{j=0}^{2m} 2^{j} \binom{2m+1}{j} B_{j}\,
\mathbb{C}_{2m+1-j}^{(1)}(x+1), \quad m \geq 1.
\end{equation*}
In particular, for~$x=0$, this produces
\begin{equation*}
\sum_{j=0}^{2m} 2^{j} \binom{2m+1}{j} B_{j}\, \mathbb{C}_{2m+1-j}^{(1)}(1)
= \frac{1- 2^{2m+2}}{m+1} B_{2m+2}, \quad m \geq 1,
\end{equation*}
where we have used the well-known formula for $E_m(0)$ (see, e.g.,~\cite{gould} (Equation~(15.47))).

\item
Regarding even-indexed Euler polynomials $E_{2m}(x)$, we checked the numerical correctness of the formula
\begin{equation*}
E_{2m}(x) = \sum_{j=0}^{2m-1} 2^{j} \binom{2m}{j} B_{j}\,
\Big( \mathbb{C}_{2m-j}^{(1)}(x+1) -  \mathbb{C}_{2m-j}^{(1)}(1) \Big), \quad m \geq 1,
\end{equation*}
fulfilling that $E_{2m}(0) =0$ for all $m \geq 1$.

\item
By introducing the generalized Euler polynomials
\begin{equation*}
E^{(k)}_{2m+1}(x) = \sum_{j=0}^{2m} 2^{j} \binom{2m+1}{j} B_{j}\,
\mathbb{C}_{2m+1-j}^{(k)}(x+1),
\end{equation*}
the formulas for $c^{(k)}_{2n+1}(x)$ and $\widehat{c}^{(k)}_{2n+1}(x)$ in Theorem \ref{th:11} can be written more compactly~as
\begin{equation*}
c^{(k)}_{2n+1}(x) = -(2n+1) \sum_{m=1}^n \frac{u(n,m)}{2m+1} E_{2m+1}^{(k)}(x+n),
\end{equation*}
and
\begin{equation*}
\widehat{c}^{(k)}_{2n+1}(x) = (2n+1) \sum_{m=1}^n \frac{u(n,m)}{2m+1} E_{2m+1}^{(k)}(x-n).
\end{equation*}
\end{itemize}

\section{Multiparameter Poly-Cauchy~Polynomials}\label{sec:7}

As mentioned in the Introduction, numerous generalizations of Cauchy numbers and polynomials have been proposed in the literature, and many of their properties have been studied in various contexts. For our purpose here, among the proposed generalizations, we highlight Cauchy numbers with a $q$ parameter~\cite{komatsu2} and shifted poly-Cauchy numbers~\cite{komatsu5}. Komatsu {\it~et~al.\/} \cite{komatsu3} went further by combining the last two types of generalizations. Specifically, for~integers $n \geq 0$, $k \geq 1$, a~positive real number $a$, and~non-zero real numbers $q$ and $l_1,\ldots, l_k$, the~authors of~\cite{komatsu3} defined poly-Cauchy numbers $c_{n,a,q,L}^{(k)}$ and $\widehat{c}_{n,a,q,L}^{(k)}$ by
\begin{equation*}
c_{n,a,q,L}^{(k)} = \int_{0}^{l_1} \ldots \int_{0}^{l_k}
(t_1 \cdots t_k)^a \prod_{j=1}^{n-1} (t_1 \cdots t_k - j q) dt_1 \cdots dt_k,
\end{equation*}
and
\begin{equation*}
\widehat{c}_{n,a,q,L}^{(k)} = (-1)^{a-1} \int_{0}^{l_1} \ldots \int_{0}^{l_k}
(-t_1 \cdots t_k)^a \prod_{j=1}^{n-1} (-t_1 \cdots t_k - j q) dt_1 \cdots dt_k,
\end{equation*}
where $L$ stands for the $k$-tuple $(l_1, \ldots, l_k)$. From~\cite{komatsu3} (Theorems 1 and 8), $c_{n,a,q,L}^{(k)}$ and $\widehat{c}_{n,a,q,L}^{(k)}$ can be expressed in terms of (unsigned) Stirling numbers of the first kind as follows:
\begin{equation}\label{shif1}
c_{n,a,q,L}^{(k)} = \sum_{m=0}^n (-q)^{n-m} \, \frac{(l_1 \cdots l_k)^{m+a}}{(m+a)^k}
\genfrac{[}{]}{0pt}{}{n}{m},
\end{equation}
and
\begin{equation}\label{shif2}
\widehat{c}_{n,a,q,L}^{(k)} = (-1)^n \sum_{m=0}^n q^{n-m} \, \frac{(l_1 \cdots l_k)^{m+a}}{(m+a)^k}
\genfrac{[}{]}{0pt}{}{n}{m}.
\end{equation}
{As} noted in~\cite{komatsu3}, if~$a = l_1 = \ldots = l_k =1$, then $c_{n,1,q,(1,\ldots,1)}^{(k)} = c_{n,q}^{(k)}$ and $\widehat{c}_{n,1,q,(1,\ldots,1)}^{(k)} = \widehat{c}_{n,q}^{(k)}$ are the poly-Cauchy numbers with a $q$ parameter, as~defined in~\cite{komatsu2}. In~contrast, if~$q = l_1 = \ldots = l_k =1$, then $c_{n,a,1,(1,\ldots,1)}^{(k)} = c_{n,a}^{(k)}$ and $\widehat{c}_{n,a,1,(1,\ldots,1)}^{(k)} = \widehat{c}_{n,a}^{(k)}$ are shifted poly-Cauchy numbers, as~defined in~\cite{komatsu5}.

In line with the above definition of $c_{n,a,q,L}^{(k)}$ and $\widehat{c}_{n,a,q,L}^{(k)}$, we now introduce a kind of multiparameter poly-Cauchy polynomials, extending the original concept (see~\cite{gomaa,young,guettal} for a general overview). In~concrete terms, we consider the following multiparameter poly-Cauchy polynomials of the first and second kind, $c^{(k)}_{n,a,q,L,y}(x)$ and $\widehat{c}^{(k)}_{n,a,q,L,y}(x)$, defined~by
\begin{equation*}
c^{(k)}_{n,a,q,L,y}(x) = \int_0^{l_1} \ldots \int_0^{l_k} (t_1 \cdots t_k -x)^{a-1}
\prod_{j=0}^{n-1}(t_1 \cdots t_k -x -y -j q) dt_1 \cdots dt_k,
\end{equation*}
and
\begin{equation*}
\widehat{c}^{(k)}_{n,a,q,L,y}(x) = (-1)^{a-1} \int_0^{l_1} \ldots \int_0^{l_k} (x - t_1 \cdots t_k)^{a-1}
\prod_{j=0}^{n-1}(x +y - t_1 \cdots t_k -j q) dt_1 \cdots dt_k,
\end{equation*}
which, in~addition to $n$, $k$, $a$, $q$, and~$l_1,\ldots, l_k$, involve the real parameters $x$ and $y$.

As the following theorem shows, $c^{(k)}_{n,a,q,L,y}(x)$ and $\widehat{c}^{(k)}_{n,a,q,L,y}(x)$ can be determined for positive integers $a$ using (bivariate) Stirling polynomials of the first kind
\begin{equation}\label{bivar}
\genfrac{[}{]}{0pt}{}{n}{m}_{(y,q)} = \sum_{i=0}^{n-m} \binom{i+m}{m}\genfrac{[}{]}{0pt}{}{n}{i+m}
y^i q^{n-m-i}, \quad 0 \leq m \leq n,
\end{equation}
and augmented polynomials
\begin{equation}\label{def2}
\mathbb{C}_j^{(k)}(x;L) = \begin{cases}
  l_1 \cdots l_k, & \text{for $j=0$;} \\[1mm]
  \displaystyle{\sum_{i=0}^{j} \frac{(-1)^i}{(i+1)^k} (l_1 \cdots l_k )^{i+1} \binom{j}{i} x^{j-i}},
& \text{for $j \geq 1$.}
\end{cases}
\end{equation}

\begin{theorem}\label{th:12}
For integers $n \geq 0$, $k \geq 1$, $a \geq 1$, and~arbitrary $q$, $l_1, \ldots, l_k$, $x$, and~$y$ (with $l_1, \ldots, l_k \neq 0$), we have
\begin{equation}\label{para1}
c^{(k)}_{n,a,q,L,y}(x) = (-1)^{a-1} \sum_{m=0}^n (-1)^n \genfrac{[}{]}{0pt}{}{n}{m}_{(y,q)}
\mathbb{C}_{m+a-1}^{(k)}(x;L),
\end{equation}
and
\begin{equation}\label{para2}
\widehat{c}^{(k)}_{n,a,q,L,y}(x) = (-1)^{a-1} \sum_{m=0}^n (-1)^{n-m} \genfrac{[}{]}{0pt}{}{n}{m}_{(-y,q)}
\mathbb{C}_{m+a-1}^{(k)}(x;L),
\end{equation}
where $\genfrac{[}{]}{0pt}{}{n}{m}_{(y,q)}$ and $\mathbb{C}_m^{(k)}(x;L)$ are the polynomials defined in \eqref{bivar} and \eqref{def2}, respectively.
\end{theorem}
\begin{proof}
Using the generating function for $\genfrac{[}{]}{0pt}{}{n}{m}_{(y,q)}$ (cf.\ \cite{mitri}, Equation~(1.1))
\begin{equation*}
\prod_{j=0}^{n-1} (x - (y+j q)) = \sum_{m=0}^n (-1)^{n-m} \genfrac{[}{]}{0pt}{}{n}{m}_{(y,q)} x^m,
\end{equation*}
it follows that
\begin{align*}
(t_1 \cdots t_k -x)^{a-1} & \prod_{j=0}^{n-1} (t_1 \cdots t_k -x - (y+j q)) \\
& = \sum_{m=0}^n (-1)^{n-m} \genfrac{[}{]}{0pt}{}{n}{m}_{(y,q)} (t_1 \cdots t_k - x)^{m+a-1} \\
& = (-1)^{a-1} \sum_{m=0}^n (-1)^{n} \genfrac{[}{]}{0pt}{}{n}{m}_{(y,q)} \sum_{i=0}^{m+a-1}
(-1)^i \binom{m+a-1}{i} x^{m+a-1-i} (t_1 \cdots t_k)^i.
\end{align*}
Thus,
\begin{align*}
c^{(k)}_{n,a,q,L,y}(x) & = (-1)^{a-1} \sum_{m=0}^n (-1)^{n} \genfrac{[}{]}{0pt}{}{n}{m}_{(y,q)} \\
& \qquad \times \sum_{i=0}^{m+a-1} (-1)^i \binom{m+a-1}{i} x^{m+a-1-i}
\int_0^{l_1} \ldots \int_0^{l_k} t_1^i \cdots t_k^i dt_1 \cdots dt_k,
\end{align*}
which, after~performing the multiple integral, can be expressed in the form of Equation~\eqref{para1}. Identity \eqref{para2} is proven similarly.
\end{proof}

We present a few comments are in relation to $c^{(k)}_{n,a,q,L,y}(x)$ and $\widehat{c}^{(k)}_{n,a,q,L,y}(x)$.
\begin{itemize}
\item
For fixed $q$, $L$, and~$y$, $c^{(k)}_{n,a,q,L,y}(x)$ and $\widehat{c}^{(k)}_{n,a,q,L,y}(x)$ are polynomials in $x$ of degree $n+a-1$.

\item
For $a = q = l_1 = \ldots = l_k =1$ and $y=0$, $c^{(k)}_{n,a,q,L,y}(x)$ and $\widehat{c}^{(k)}_{n,a,q,L,y}(x)$ reduce to the ordinary poly-Cauchy polynomials
\begin{equation*}
c^{(k)}_{n}(x) = \sum_{m=0}^n (-1)^{n} \genfrac{[}{]}{0pt}{}{n}{m} \mathbb{C}_{m}^{(k)}(x),
\end{equation*}
and
\begin{equation*}
\widehat{c}^{(k)}_{n}(x) = \sum_{m=0}^n (-1)^{n-m} \genfrac{[}{]}{0pt}{}{n}{m} \mathbb{C}_{m}^{(k)}(x),
\end{equation*}
which correspond to the formulas in \eqref{back1} and \eqref{back2}, respectively.

\item
For $x=0$, $c^{(k)}_{n,a,q,L,y}(x)$ and $\widehat{c}^{(k)}_{n,a,q,L,y}(x)$ become
\begin{equation*}
c^{(k)}_{n,a,q,L,y}(0) \equiv c^{(k)}_{n,a,q,L,y} = \sum_{m=0}^n (-1)^{n-m}
\genfrac{[}{]}{0pt}{}{n}{m}_{(y,q)} \frac{(l_1 \cdots l_k)^{m+a}}{(m+a)^k},
\end{equation*}
and
\begin{equation*}
\widehat{c}^{(k)}_{n,a,q,L,y}(0) \equiv \widehat{c}^{(k)}_{n,a,q,L,y} = (-1)^n \sum_{m=0}^n
\genfrac{[}{]}{0pt}{}{n}{m}_{(-y,q)} \frac{(l_1 \cdots l_k)^{m+a}}{(m+a)^k}.
\end{equation*}
Furthermore, since $\genfrac{[}{]}{0pt}{}{n}{m}_{(0,q)} = \genfrac{[}{]}{0pt}{}{n}{m} q^{n-m}$, $c^{(k)}_{n,a,q,L,y}$ and $\widehat{c}^{(k)}_{n,a,q,L,y}$ reduce to \eqref{shif1} and \eqref{shif2} when $y=0$. On~the other hand, if~$a = q = l_1 = \ldots = l_k =1$, $c^{(k)}_{n,a,q,L,y}$ and $\widehat{c}^{(k)}_{n,a,q,L,y}$ are the same, respectively, as~\eqref{pro41} and \eqref{pro42} (after renaming $y$ to $x$).

\item
From the definition of $c^{(k)}_{n,a,q,L,y}(x)$ and $\widehat{c}^{(k)}_{n,a,q,L,y}(x)$, it can be seen that, for~$a=1$, they remain invariant under the interchange of $x$ and $y$. So, for~example,
\begin{equation*}
c^{(3)}_{4,1,-3,(1,1,\frac{1}{2}),-\frac{3}{2}}(\sqrt{5}) = c^{(3)}_{4,1,-3,(1,1,\frac{1}{2}),\sqrt{5}}(-\tfrac{3}{2})
= \frac{114177911}{144000} - \frac{284203 \,\sqrt{5}}{768},
\end{equation*}
and
\begin{equation*}
\widehat{c}^{(3)}_{4,1,-3,(1,1,\frac{1}{2}),-\frac{3}{2}}(\sqrt{5}) = \widehat{c}^{(3)}_{4,1,-3,(1,1,\frac{1}{2}),\sqrt{5}}(-\tfrac{3}{2})
= \frac{14046697}{288000} + \frac{10805 \,\sqrt{5}}{768}.
\end{equation*}
Moreover, $c^{(k)}_{n,a,q,L,y}(x)$ and $\widehat{c}^{(k)}_{n,a,q,L,y}(x)$ are related by
\begin{equation*}
c^{(k)}_{n,a,q,L,y}(x) = (-1)^n \, \widehat{c}^{(k)}_{n,a,-q,L,y}(x).
\end{equation*}

\item
Let $\{ S^1, S^2 \}$ be the Stirling pair~\cite{hsu} composed of $S^1 = \genfrac{[}{]}{0pt}{}{n}{m}_{(y,q)}$ and $S^2 = \genfrac{\{ }{\} }{0pt}{}{n}{m}_{(y,q)}$, with~\begin{equation*}
\genfrac{\{}{\}}{0pt}{}{n}{m}_{(y,q)} = \frac{1}{m! q^m} \sum_{l=0}^m (-1)^{m-l} \binom{m}{l} (y + l q)^n.
\end{equation*}
If we define multiparameter poly-Bernoulli polynomials $\mathfrak{B}^{(k)}_{n,a,q,L,y}(x)$ by
\begin{equation*}
\mathfrak{B}^{(k)}_{n,a,q,L,y}(x) = (-1)^n \sum_{m=0}^n m! \genfrac{\{}{\}}{0pt}{}{n}{m}_{(y,q)}
\mathbb{C}_{m+a-1}^{(k)}(x;L),
\end{equation*}
we have the following relationships between $\mathfrak{B}^{(k)}_{n,a,q,L,y}(x)$ and multiparameter poly-Cauchy polynomials  $c^{(k)}_{n,a,q,L,y}(x)$ and $\widehat{c}^{(k)}_{n,a,q,L,y}(x)$:
\begin{align*}
\mathfrak{B}^{(k)}_{n,a,q,L,y}(x) & = (-1)^{n+a-1} \sum_{m=0}^n \sum_{l=0}^m  (-1)^{m} m!
\genfrac{\{}{\}}{0pt}{}{n}{m}_{(y,q)}
\genfrac{\{}{\}}{0pt}{}{m}{l}_{(y,q)} c_{l,a,q,L,y}^{(k)}(x), \\
\mathfrak{B}^{(k)}_{n,a,q,L,y}(x) & = (-1)^{n+a-1} \sum_{m=0}^n \sum_{l=0}^m  m!
\genfrac{\{}{\}}{0pt}{}{n}{m}_{(y,q)}
\genfrac{\{}{\}}{0pt}{}{m}{l}_{(-y,q)} \widehat{c}_{l,a,q,L,y}^{(k)}(x), \\
c_{n,a,q,L,y}^{(k)}(x) & = (-1)^{n+a-1} \sum_{m=0}^n \sum_{l=0}^m  \frac{(-1)^{m}}{m!}
\genfrac{[}{]}{0pt}{}{n}{m}_{(y,q)}
\genfrac{[}{]}{0pt}{}{m}{l}_{(y,q)} \mathfrak{B}_{l,a,q,L,y}^{(k)}(x), \\
\widehat{c}_{n,a,q,L,y}^{(k)}(x) & = (-1)^{n+a-1} \sum_{m=0}^n \sum_{l=0}^m \frac{1}{m!} \genfrac{[}{]}{0pt}{}{n}{m}_{(-y,q)}
\genfrac{[}{]}{0pt}{}{m}{l}_{(y,q)} \mathfrak{B}_{l,a,q,L,y}^{(k)}(x),
\end{align*}
which reduce to \eqref{inform} when $a =1$, $q=1$, $l_1 = \ldots = l_k =1$, and~$y =0$.

\end{itemize}

\section{Conclusions}\label{sec:8}

In this paper, we mainly report a variety of (known and novel) formulas and identities involving Cauchy and poly-Cauchy numbers and polynomials, ordinary and generalized Stirling numbers, binomial coefficients, central factorial numbers, Euler polynomials, $r$-Whitney numbers, hyperharmonic polynomials, and~Bernoulli numbers and polynomials. Furthermore, we obtain several recurrence and higher-order derivative formulas (amongst others) for Cauchy and poly-Cauchy polynomials. Moreover, we develop an extended version of poly-Cauchy polynomials accounting, in~particular, for~shifted poly-Cauchy numbers and the polynomials with a $q$ parameter.

In conclusion, it is our intention that this paper may serve as a useful compilation of facts and results about the Cauchy numbers and polynomials. We also hope that it may stimulate further research into these fascinating objects and their connections with other types of numbers and polynomials such as, for~example, the~number sequence $\mathcal{A}_x(n,m)$~\cite{chen}, the~poly-Daehee numbers and polynomials~\cite{daehee1,daehee2,daehee3}, the~generalized $r$-Whitney numbers~\mbox{\cite{whit2,whit3}}, and~the generalized harmonic numbers~\cite{cheon,cheon2,harm2,kargin2}.

In this respect, it is worth noting that the equations in Proposition \ref{prop:2} can alternatively be expressed in the form
\begin{equation*}
\sum_{m=0}^n (-1)^{m} H_{m}(x+1) \frac{c_{n-m}(y)}{(n-m)!} = \binom{x-y}{n},
\end{equation*}
and
\begin{equation*}
\sum_{m=0}^n (-1)^{m} H_{m}(x+2) \frac{\widehat{c}_{n-m}(y)}{(n-m)!} = \binom{x+y}{n},
\end{equation*}
where $H_m(x)$ are the so-called harmonic polynomials in $x$ of degree $m$, defined by the generating function~\cite{cheon2} (Equation~(28))
\begin{equation*}
\sum_{m=0}^{\infty} H_m(x) t^m = - \frac{\ln (1-t)}{t(1-t)^{1-x}},
\end{equation*}
with $H_m(0) = H_{m+1}$ for all $m \geq 0$. Furthermore, it can be shown that, for~$n \geq 1$,
\begin{equation*}
\frac{\widehat{c}_n(x)}{n!} = \binom{x}{n} - \sum_{m=0}^{n-1} \frac{(-1)^m \, c_{n-m}}{(n-m-1)!}
\, H_{m}(x+1).
\end{equation*}
{Upon} comparing this equation with \eqref{hyp5}, we obtain the identity
\begin{equation*}
\sum_{m=0}^n \frac{(-1)^m \, c_{m}(-x)}{m!(n-m+1)(n-m+2)}
= \sum_{m=0}^n \frac{(-1)^{n-m} \, c_{n+1-m}}{(n-m)!} \, H_m(x), \quad n \geq 0.
\end{equation*}
{In particular,} for~$x=0$, we have
\begin{equation*}
\sum_{m=0}^n \frac{(-1)^{m} \, c_{n-m}}{(n-m)!(m+1)(m+2)}
= \sum_{m=0}^n \frac{(-1)^{m} \, c_{n+1-m}}{(n-m)!} \, H_{m+1}, \quad n \geq 0.
\end{equation*}

\begin{remark}
All properties and results of this paper have been tested using the computer algebra system {\it Mathematica}.
\end{remark}

\section*{Acknowledgments}

The author thanks the anonymous referees for their careful reading of the manuscript and for their useful comments and suggestions for improving the presentation of this study. He also thanks \mbox{Jos\'{e} A.\ Pocino} for providing him with some of the references listed below.

\end{document}